\documentclass[11pt,twoside]{amsart}
\usepackage[margin=1.25in]{geometry}
\usepackage{amssymb}
\usepackage{mathtools}
\usepackage{tikz,xcolor}
\usepackage{ytableau}
\usepackage[shortlabels]{enumitem}
\usepackage{url}
\usepackage{subcaption}
\usepackage{hyperref}

\title{An RSK correspondence for cylindric tableaux}
\author{Alexander Dobner}

\newcommand{\Z}{\mathbb{Z}}
\newcommand{\N}{\mathbb{N}}
\newcommand{\R}{\mathbb{R}}
\newcommand{\sh}{\mathrm{sh}}
\newcommand{\Rect}{\mathrm{Rect}}
\newcommand{\LNE}{\mathrm{LNE}}
\newcommand{\MCW}{\mathrm{MCW}}
\newcommand{\tr}{\mathrm{tr}}
\newcommand{\wt}{\mathrm{wt}}

\newtheorem{theorem}{Theorem}[section]
\newtheorem{corollary}[theorem]{Corollary}
\newtheorem{lemma}[theorem]{Lemma}
\newtheorem{proposition}[theorem]{Proposition}
\theoremstyle{remark}
\newtheorem{remark}[theorem]{Remark}

\begin{document}
\begin{abstract}
    This paper establishes an analogue of the Robinson--Schensted correspondence for cylindric tableaux. In particular, for any pair of positive integers $(d,L)$, we construct a bijection between permutations that avoid the patterns $d\cdots 1 (d+1)$ and $1\cdots (L+1)$ and pairs of $(d,L)$-cylindric standard Young tableaux with a common shape. This arises as a special case of a Knuth-type generalization involving cylindric semistandard tableaux and a further generalization involving oscillating tableaux. Using these results, we construct several other bijections and derive enumerative consequences involving cylindric tableaux and pattern-avoiding permutations. For example, we give an asymptotic for the number of permutations in $S_n$ that avoid the patterns $d\cdots 1 (d+1)$ and $1\cdots (L+1)$ as $n\to\infty$.
\end{abstract}

\maketitle

\section{Introduction}

\subsection{Overview}

The \emph{Robinson--Schensted (RS) correspondence} \cite{schensted1961} is a bijection between permutations and pairs of standard Young tableaux of the same shape. It has an extension known as the \emph{Robinson--Schensted--Knuth (RSK) correspondence} \cite{knuth1970} that gives a bijection between matrices with nonnegative entries and pairs of semistandard Young tableaux of the same shape. 

In various affine and quantum settings, \emph{cylindric tableaux} play a parallel role to that of classical tableaux. In particular, these cylindric analogues arise in the representation theory of Hecke algebras \cite{wenzl1988,suzuki2005v} and in the combinatorics of fusion coefficients \cite{morse2012s}. Postnikov \cite{postnikov2005} used cylindric tableaux to define cylindric Schur functions, which he studied in connection with Gromov--Witten invariants.  Cylindric Schur functions have subsequently been studied by many others \cite{mcnamara2006,lam2006,lee2019,korff2020p,huh2025kkc}.

Unlike classical tableaux, which are drawn in the plane, cylindric tableaux are naturally embedded on the surface of a cylinder (see Figure \ref{fig:cssyt}). The underlying combinatorics of cylindric tableaux are less well understood than their classical counterparts. Goodman and Wenzl \cite{goodman1990w} posed the question of whether there is a Robinson--Schensted-type correspondence for these objects. 

The goal of this paper is to develop analogues of the RS and RSK correspondences for cylindric tableaux. The simplest case can be stated as follows.

\begin{theorem}[Cylindric RS correspondence]\label{thm:cRS}
    For all $n,d,L \in \N$, there is an explicit bijective correspondence between
    \begin{enumerate}
        \item permutations in $S_n$ that avoid the patterns $d\cdots 1 (d+1)$ and $1\cdots (L+1)$, and
        \item pairs $(P,Q)$ such that $P$ and $Q$ are $(d,L)$-cylindric standard Young tableaux of size $n$ with the same shape.
    \end{enumerate}
    Under this bijection, taking the inverse of the permutation corresponds to interchanging $P$ and $Q$.
\end{theorem}

This result is a special case of a more general Knuth-type correspondence involving cylindric semistandard tableaux. The relevant definitions and precise statements of these results will be given in Sections \ref{sec:definitions}, \ref{sec:mainresults}, and \ref{sec:corollaries}.

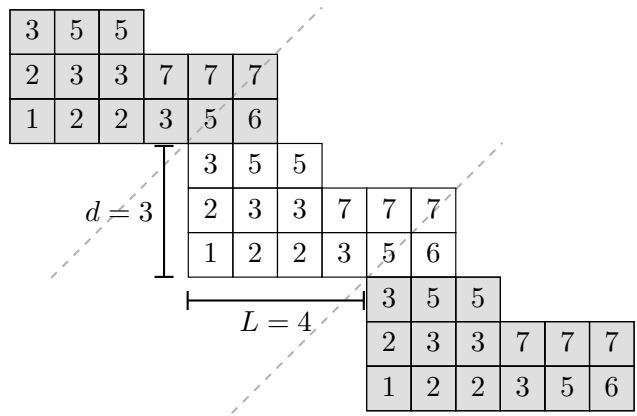
\begin{figure}
    \centering
    \begin{tikzpicture}
        \node at (0,0) {\ydiagram[*(gray!25)]{3,6,6,0,0,0,8+3,8+6,8+6}};
        \draw[-,gray!60,thick,dashed] (-6*.6,-.6*2+.3) -- (-0*0.6,.6*4+.3);
        \draw[-,gray!60,thick,dashed] (-2*.6,-5*.6+.3) -- (4*.6,1*.6+.3);
        \draw[|-|,thick] (-3.5*.6,-1.5*.6) -- (-3.5*.6,1.5*.6-.03) node[midway,left] {$d=3$};
        \draw[|-|,thick] (-3*.6,-2*.6) -- (1*.6-.03,-2*.6) node[midway,below] {$L=4$};
        \node {\ytableaushort{355,233777,122356,
    \none\none\none\none 355,
    \none\none\none\none 233777,
    \none\none\none\none 122356,
    \none\none\none\none\none\none\none\none 355,
    \none\none\none\none\none\none\none\none 233777,
    \none\none\none\none\none\none\none\none 122356}};
    \end{tikzpicture}
    \caption{In white: a semistandard Young tableau. This tableau is $(3,4)$-cylindric since the entries are strictly increasing along columns even after adding shifted copies. The picture can be drawn on a cylinder by gluing the two dashed lines together.}
    \label{fig:cssyt}
\end{figure}

The cylindric RS correspondence can be viewed as a two-parameter deformation of the classical case. Indeed, in the $d\to\infty$ and $L\to\infty$ limit, the cylindric RS correspondence coincides with the usual RS correspondence. In the case where $d$ is fixed and $L\to\infty$, we get a bijection that we call the \emph{$d$-RS correspondence} between permutations that avoid the pattern $d\cdots 1 (d+1)$ and pairs of standard Young tableaux with a common shape and at most $d$ rows. It is possible to define these correspondences in terms of a modified Schensted insertion algorithm (the idea being that whenever a number is bumped out of row $d$, it is reinserted into row 1 rather than into row $d+1$). However, we will instead use Fomin's growth diagrams \cite{fomin1986}. This makes the proofs more transparent, and it yields further generalizations involving oscillating tableaux with no additional effort.

As is typical in proofs involving growth diagrams, we will treat tableaux as sequences of partitions rather than fillings of Young diagrams. A notable (but completely elementary) feature of our work in this regard is that we give a definition for cylindric semistandard tableaux in terms of a relation that we call \emph{$(d,L)$-interlacing}. We also introduce a new statistic on tableaux that we call the \emph{minimum cylindric width}. We show that under the $d$-RS correspondence, the length of the longest increasing subsequence of a permutation is equal to the minimum cylindric width statistic on the tableau side. This is in contrast to the usual RS correspondence where the length of the longest increasing subsequence is equal to the length of the first row of the tableaux.

Although the connection between cylindric tableaux and pattern-avoiding permutations that we establish in this paper is new, there are some natural predecessors of our work in the literature. In particular, Bloom and Saracino \cite{bloom2012s} defined a certain map using growth diagrams which is essentially the $d$-RS correspondence. Their work does not involve cylindric tableaux, however. On the other hand, Elizalde \cite{elizalde2025} used growth diagrams to reformulate a correspondence of Neyman \cite{neyman2015} which involves cylindric tableaux but not permutations. Our work shows that these growth diagrams are closely related. In Section \ref{sec:connections} we will give a direct comparison of the local rules used in our growth diagrams and those used by Elizalde and Bloom and Saracino. In Section \ref{sec:skewdRSK} we will generalize Neyman's correspondence.

We conclude this overview by mentioning that we were led to discover Theorem~\ref{thm:cRS} empirically while investigating certain random matrix ensembles. We describe this story in more detail in Section~\ref{sec:asymptotics}. There, we will use a result from our companion paper \cite{dobner2026} together with the cylindric RS correspondence to derive an asymptotic formula for the number of permutations avoiding the patterns $d\cdots 1 (d+1)$ and $1\cdots (L+1)$. Such results are of interest in the theory of pattern-avoiding permutations, and we are not aware of any previous results on this particular enumeration problem.

\subsection{Outline}

The remainder of this paper is organized as follows. In Section~\ref{sec:definitions} we establish our notation and the necessary background on partitions, Young diagrams, and pattern avoidance, and we formally define various notions of cylindric tableaux. In Section~\ref{sec:mainresults} we define growth diagrams and state our main results. In Section~\ref{sec:corollaries} we derive some immediate corollaries of our main results (in particular, the cylindric RS and cylindric RSK correspondences). In Sections~\ref{sec:dRSKchains} and \ref{sec:dRSKgrowth} we give the proofs of the main results. In Section~\ref{sec:skewdRSK} we show how our growth diagrams can be modified to construct a different correspondence involving skew cylindric tableaux. This is a generalization of a correspondence of Neyman \cite{neyman2015}.

The remaining sections give various auxiliary consequences of the correspondences described in Sections 3--7.  In Section~\ref{sec:dual}, we discuss cylindric conjugation and cylindric row-strict tableaux, and how conjugation can be combined with the correspondences. In Section~\ref{sec:fillings_cors}, we derive some additional bijections involving fillings of Young diagrams. In Section~\ref{sec:permutations}, we prove several results on permutations in $S_n$ that avoid the patterns $d\cdots 1 (d+1)$ and $1\cdots (L+1)$ such as an asymptotic formula for the number of such permutations as $n\to\infty$. In Section~\ref{sec:connections}, we discuss how the local rules used in our growth diagrams are related to local rules used by Bloom and Saracino \cite{bloom2012s}, and Elizalde \cite{elizalde2025}. In Section \ref{sec:continuous}, we describe continuous piecewise-linear generalizations of our correspondences.

\subsection{Acknowledgements}
The author thanks George H. Seelinger for many helpful discussions on the Robinson--Schensted correspondence and related topics. The author also thanks Jeff Lagarias for his encouragement in pursuing the results that led to this paper.

\section{Definitions} \label{sec:definitions}
\subsection{Pattern-avoiding permutations} \label{sec:pattern_avoidance}
Let $\tau \in S_m$ and $\pi \in S_n$ be permutations. We say that \emph{$\pi$ contains the pattern $\tau$} if the permutation matrix of $\pi$ contains the permutation matrix of $\tau$ as a submatrix. If this does not occur, we say that \emph{$\pi$ avoids the pattern $\tau$}. It is standard to write the pattern $\tau$ using one-line notation; e.g. $\tau=1324$. 

The patterns that will be relevant to us in this paper are the ones appearing in Theorem~\ref{thm:cRS}. These are $d\cdots 1(d+1)$ and $1\cdots (L+1)$ for positive integers $d$ and $L$. It is easy to describe pattern avoidance for these directly. For example, $\pi$ avoids the pattern $1\cdots (L+1)$ if and only if $\pi$ contains no increasing subsequence of length $L+1$. Similarly, $\pi$ avoids the pattern $d\cdots 1(d+1)$ if and only if there does not exist a subsequence of $\pi$ that consists of $d$ elements in decreasing order followed by one additional element that is larger than the first $d$ elements.

\subsection{Partitions, Young diagrams, and fillings} 
A \emph{partition} is a finite sequence of positive integers $\lambda = (\lambda_1,\ldots,\lambda_m)$ such that $\lambda_1 \geq \lambda_2 \geq \cdots \geq \lambda_m$. The $\lambda_i$'s are called the \emph{parts} of $\lambda$, and the sum of the parts is called the \emph{size} of $\lambda$, denoted $|\lambda|$. The number of parts of $\lambda$ is called the \emph{length} of $\lambda$, denoted $\ell(\lambda)$. The \emph{empty partition}, denoted $\emptyset$, is the unique partition of length zero. In general we will use lowercase Greek letters to denote partitions and subscripts to denote the individual parts of a partition. It is also convenient to extend the subscript notation by defining $\lambda_i=0$ whenever $i>\ell(\lambda)$. Given two partitions $\lambda$ and $\mu$, we say $\lambda \subseteq \mu$ if $\lambda_i \leq \mu_i$ for all $i$. Naturally this implies that $|\lambda| \leq |\mu|$ and $\ell(\lambda) \leq \ell(\mu)$.

Every partition $\lambda$ has an associated pictorial representation called a \emph{Young diagram}. This is a left-justified arrangement of rows of squares such that the $i$th row consists of $\lambda_i$ squares. For concreteness, we define Young diagrams to lie in the first quadrant of the Cartesian plane, and we take the squares to be unit squares with vertices at integer lattice points. The first row of squares is adjacent to the $x$-axis, the second row is directly above the first row, and so on. We will use capital letters to denote Young diagrams. Given a Young diagram $F$, we will refer to the unit squares that make up $F$ as \emph{cells}. We will refer to the vertices of all the cells of $F$ as the \emph{lattice points of $F$}, and we will refer to the sides of the cells as the \emph{edges of $F$}. Given any lattice point $p$ in the first quadrant, we will write $\Rect_p$ to denote the rectangular Young diagram whose upper-right corner is $p$.

\begin{figure}
\centering
\begin{tikzpicture}
    \begin{scope}[scale=0.9]
        % Axes
        \draw[->] (-0.5,0) -- (5,0);
        \draw[->] (0,-0.5) -- (0,4);
        \node[right] at (5,0) {$x$};
        \node[above] at (0,4) {$y$};

        % Row 1: 4 squares
        \draw (0,0) rectangle (4,1);
        \draw (1,0) -- (1,1);
        \draw (2,0) -- (2,1);
        \draw (3,0) -- (3,1);
        
        % Row 2: 3 squares
        \draw (0,1) rectangle (3,2);
        \draw (1,1) -- (1,2);
        \draw (2,1) -- (2,2);
        
        % Row 3: 1 square
        \draw (0,2) rectangle (1,3);
        
        % Draw outer boundary in bold
        \draw[line width=2.4pt] (0,3) -- (1,3) node[midway,above,inner sep=1pt,font=\scriptsize] {$-$} -- (1,2) node[midway,right,inner sep=1pt,font=\scriptsize] {$+$} -- (2,2) node[midway,above,inner sep=1pt,font=\scriptsize] {$-$} -- (3,2) node[midway,above,inner sep=1pt,font=\scriptsize] {$-$} -- (3,1) node[midway,right,inner sep=1pt,font=\scriptsize] {$+$} -- (4,1) node[midway,above,inner sep=1pt,font=\scriptsize] {$-$} -- (4,0) node[midway,right,inner sep=1pt,font=\scriptsize] {$+$};
        
        % Draw lattice points
        \foreach \x in {0,...,4} {
            \foreach \y in {0,...,3} {
                \ifnum\x=0
                    \fill (\x,\y) circle (2pt);
                \else\ifnum\y=0
                    \fill (\x,\y) circle (2pt);
                \else\ifnum\x=1
                    \ifnum\y<4
                        \fill (\x,\y) circle (2pt);
                    \fi
                \else\ifnum\x<4
                    \ifnum\y<3
                        \fill (\x,\y) circle (2pt);
                    \fi
                \else
                    \ifnum\y<2
                        \fill (\x,\y) circle (2pt);
                    \fi
                \fi\fi\fi\fi
            }
        }
    \end{scope}
\end{tikzpicture}
\caption{The Young diagram of the partition $(4,3,1)$. The lattice points of the diagram are depicted as black dots. The outer boundary (shown in bold) is encoded by the type sequence ${+}{-}{+}{-}{-}{+}{-}$. }
\label{fig:young_diagram}
\end{figure}
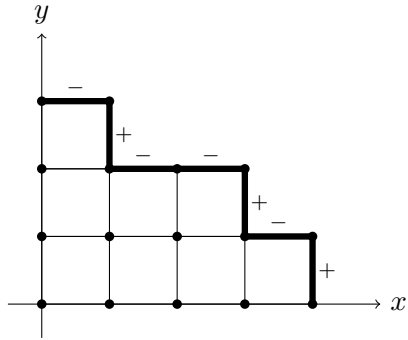

The boundary of any nonempty Young diagram consists of a line segment on the $y$-axis, a line segment on the $x$-axis, and a lattice path going from the positive $x$-axis to the positive $y$-axis consisting of unit steps up and to the left. We call this lattice path the \emph{outer boundary} of the Young diagram. We will encode the outer boundary using what we call a \emph{type sequence}, which is a finite word $w=w_1\cdots w_r$ in the alphabet $\{+, -\}$. The $i$th letter of this sequence is `$+$' if the $i$th step of the outer boundary is a step up, and it is `$-$' if the $i$th step is a step to the left. See Figure \ref{fig:young_diagram} for an example.

A \emph{filling} of a Young diagram is an assignment of nonnegative integers to the cells of the diagram. This is represented pictorially by writing each integer (called an \emph{entry}) inside the corresponding cell. Given a Young diagram and a filling, a \emph{NE-chain} is a sequence of nonzero entries in which each entry is weakly above and weakly to the right of the previous one. The \emph{length} of a NE-chain is the sum of the entries in the chain. A \emph{se-chain} is a sequence of nonzero entries in which each entry lies strictly below and strictly to the right of the previous one. The \emph{length} of a se-chain is the number of entries in the chain. (Note the difference in definitions.) See Figure \ref{fig:NEchain} and \ref{fig:sechain} for examples.

Given a positive integer $d$, we say that a filling \emph{contains the pattern $d \cdots 1(d+1)$} if there exists a sequence of $d+1$ nonzero entries such that the first $d$ of them form a se-chain, and the last one is strictly above and strictly to the right of those in the se-chain. If no such sequence exists, then we say that the filling \emph{avoids the pattern $d \cdots 1(d+1)$}. See Figure \ref{fig:pattern} for an example.

A Young diagram with a filling can be viewed as a generalization of a matrix with nonnegative integer entries. Indeed, given such a matrix we can identify it with the filling of a rectangular Young diagram, although we must reflect the entries vertically to match our convention for indexing the rows of a Young diagram. A permutation matrix is a special case of this, and we can see that under this identification, the definition of pattern avoidance for permutations is consistent with the one we have just given for fillings avoiding the pattern $d \cdots 1(d+1)$.

\begin{figure}
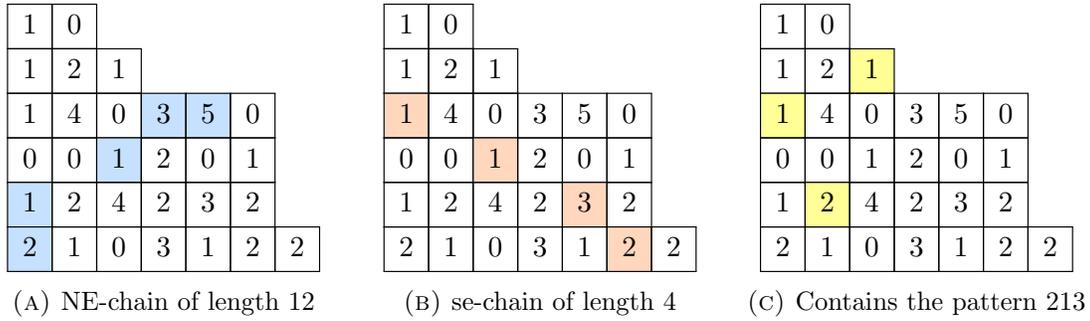

    \centering
    \definecolor{necolor}{RGB}{198,224,255}
    \definecolor{secolor}{RGB}{255,215,189}
    \definecolor{patterncolor}{RGB}{255, 255, 150}
    
    \begin{subfigure}{0.32\textwidth}
        \centering
        \begin{ytableau}
            1 & 0 \\
            1 & 2 & 1 \\
            1 & 4 & 0 & *(necolor) 3 & *(necolor) 5 & 0\\
            0 & 0 & *(necolor) 1 & 2 & 0 & 1 \\
            *(necolor) 1 & 2 & 4 & 2 & 3 & 2 \\
            *(necolor) 2 & 1 & 0 & 3 & 1 & 2 & 2
        \end{ytableau}
        \caption{NE-chain of length $12$}
        \label{fig:NEchain}
    \end{subfigure}
    \begin{subfigure}{0.32\textwidth}
        \centering
        \begin{ytableau}
            1 & 0 \\
            1 & 2 & 1 \\
            *(secolor) 1 & 4 & 0 & 3 & 5 & 0\\
            0 & 0 & *(secolor) 1 & 2 & 0 & 1 \\
            1 & 2 & 4 & 2 & *(secolor) 3 & 2 \\
            2 & 1 & 0 & 3 & 1 & *(secolor) 2 & 2
        \end{ytableau}
        \caption{se-chain of length 4}
        \label{fig:sechain}
    \end{subfigure}
    \begin{subfigure}{0.32\textwidth}
        \centering
        \begin{ytableau}
            1 & 0 \\
            1 & 2 & *(patterncolor) 1 \\
            *(patterncolor) 1 & 4 & 0 & 3 & 5 & 0\\
            0 & 0 & 1 & 2 & 0 & 1 \\
            1 & *(patterncolor) 2 & 4 & 2 & 3 & 2 \\
            2 & 1 & 0 & 3 & 1 & 2 & 2
        \end{ytableau}
        \caption{Contains the pattern $213$}
        \label{fig:pattern}
    \end{subfigure}
    
    \caption{A filling with highlighted examples of a NE-chain, a se-chain, and an instance of the pattern $213$. This filling avoids the pattern $3 2 1 4$.}
\end{figure}

\subsection{Semistandard Young tableaux}
Given two partitions $\alpha$ and $\beta$, we say that \emph{$\alpha$ interlaces $\beta$}, denoted $\alpha \prec \beta$ or $\beta \succ \alpha$, if their parts satisfy the inequalities
\[\beta_1 \geq \alpha_1 \geq \beta_2 \geq \alpha_2 \geq \beta_3 \geq \cdots.\]
A \emph{semistandard Young tableau} is a finite sequence of partitions $T=(\lambda^{(0)},\ldots,\lambda^{(k)})$ such that $\emptyset = \lambda^{(0)} \prec \lambda^{(1)} \prec \ldots \prec \lambda^{(k)}.$ The last partition $\lambda^{(k)}$ is called the \emph{shape} of $T$, denoted $\sh(T)$. The \emph{weight} of $T$, denoted $\wt(T)$, is the vector of nonnegative integers whose $i$th entry is $\wt(T)_i = |\lambda^{(i)}| - |\lambda^{(i-1)}|$ for $i=1,\ldots,k$. The \emph{size} of $T$ is the size of its shape, i.e. $|\sh(T)|$.

It is well known (see \cite[p. 5]{macdonald1995}) that any semistandard Young tableau $T$ corresponds to a filling of the Young diagram of $\sh(T)$ with positive entries that are weakly increasing along rows from left to right and strictly increasing along columns from bottom to top. To obtain this filling, the entry $i$ is placed in all the cells that are in the Young diagram of $\lambda^{(i)}$ but not in the Young diagram of $\lambda^{(i-1)}$. Hence, the number of cells with entry $i$ is equal to $\wt(T)_i$. An example is shown in Figure \ref{fig:ssyt}.

Given a positive integer $d$, we say that a partition $\lambda$ is a \emph{$d$-partition} if $\ell(\lambda)\leq d$. Given two $d$-partitions $\alpha$ and $\beta$, we say that \emph{$\alpha$ $(d,L)$-interlaces $\beta$}, denoted $\alpha \prec_{(d,L)} \beta$ or $\beta \succ_{(d,L)} \alpha$, if their parts satisfy the inequalities
\[
\beta_1 \geq \alpha_1 \geq \beta_2 \geq \alpha_2 \geq \cdots \geq \beta_d \geq \alpha_d \geq \beta_1 - L.
\]
This is equivalent to requiring $\alpha \prec \beta$ and $\beta_1 - \alpha_d \leq L$. If $\alpha$ and $\beta$ are $d$-partitions such that $\alpha \prec \beta$ or $\alpha\succ\beta$ we define their \emph{minimum cylindric width}, denoted $\MCW_d(\alpha,\beta)$, to be the smallest $L$ such that $\alpha \prec_{(d,L)} \beta$ or $\alpha \succ_{(d,L)} \beta$. That is,
\[
\MCW_d(\alpha,\beta) = \begin{cases}
\beta_1 - \alpha_d & \text{if } \alpha \prec \beta, \\
\alpha_1 - \beta_d & \text{if } \alpha \succ \beta.
\end{cases}
\]
Note that if $\alpha \prec \beta$ and $\alpha \succ \beta$ both hold then $\alpha=\beta$, so there is no ambiguity in the definition.

We call a semistandard Young tableau $T=(\lambda^{(0)},\ldots,\lambda^{(k)})$ a \emph{$d$-semistandard Young tableau} if each $\lambda^{(i)}$ is a $d$-partition (equivalently, $\sh(T)$ is a $d$-partition). Furthermore, we call $T$ a \emph{$(d,L)$-cylindric semistandard Young tableau} if $\emptyset=\lambda^{(0)} \prec_{(d,L)} \lambda^{(1)} \prec_{(d,L)} \cdots \prec_{(d,L)} \lambda^{(k)}$. If $T$ is a $d$-semistandard Young tableau, we define the \emph{minimum cylindric width of $T$}, denoted $\MCW_d(T)$, to be the smallest $L$ such that $T$ is $(d,L)$-cylindric. Equivalently,
\[
\MCW_d(T) = \max_{1 \leq i \leq k} \MCW_d(\lambda^{(i-1)},\lambda^{(i)}) = \max_{1 \leq i \leq k} (\lambda^{(i)}_1 - \lambda^{(i-1)}_d).
\]

There is an alternative characterization of $(d,L)$-cylindric semistandard Young tableaux in terms of fillings of Young diagrams. Indeed, given a $d$-semistandard Young tableau $T$, the $(d,L)$-cylindric condition may be checked in the following way. We first consider the pictorial representation of $T$ as a filling of a Young diagram (as described previously). We then augment this picture by adding a copy of the first row shifted up by $d$ units and left by $L$ units. The tableau $T$ is $(d,L)$-cylindric if and only if the entries in this augmented picture are still strictly increasing along columns and there are no gaps between cells in a column. An example is shown in Figure \ref{fig:mcw}. Equivalently, one could also take infinitely many copies of the entire picture shifted by integer multiples of $(-L,d)$ and check the same condition on the columns. This infinite periodic picture is what motivates the term ``cylindric'' in the name of these tableaux (see Figure \ref{fig:cssyt}). 

The pictorial characterization of $(d,L)$-cylindric semistandard Young tableaux described above will not be needed in the proofs of our main results. However, it is useful for some purposes, such as for defining a conjugation operation on cylindric tableaux. We will discuss this further in Section~\ref{sec:dual}.

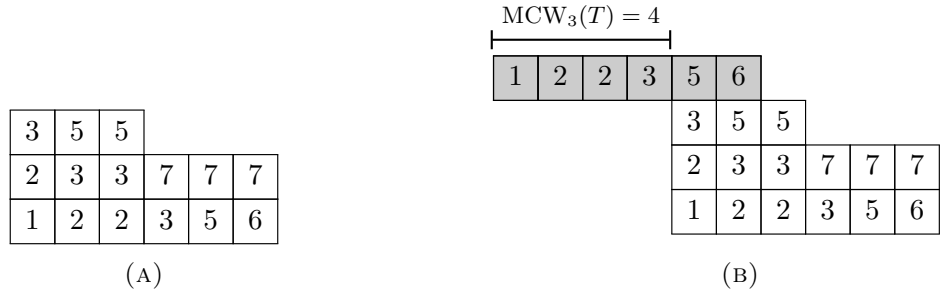
\begin{figure}
    \centering
    \begin{subfigure}{0.48\textwidth} 
        \centering
        \ytableaushort{355,233777,122356}
        \caption{}
        \label{fig:ssyt}
    \end{subfigure}
    \hfill
    \begin{subfigure}{0.48\textwidth}
        \centering
        \begin{tikzpicture}
            \node at (0,0) {\ytableaushort{122356,
                                \none\none\none\none 355,
                                \none\none\none\none 233777,
                                \none\none\none\none 122356}* [*(gray!40)]{6}};
            \draw[|-|,thick] (-3,1.4) -- (-.6,1.4) node[midway,above] {\footnotesize $\MCW_3(T)=4$};
            \node at (3.6,0) {}; % to shift the diagram to the left
        \end{tikzpicture}
        \caption{}
        \label{fig:mcw}
    \end{subfigure}
    \caption{(\textsc{a}) A depiction of the semistandard Young tableau $T=$ \\ \centerline{
    $\big(\emptyset \prec (1) \prec (3,1) \prec (4,3,1) \prec (4,3,1) \prec (5,3,3) \prec (6,3,3) \prec (6,6,3)\big)$.} \\
    In this case $\wt(T)=(1,3,4,0,3,1,3)$, and $T$ is a $d$-semistandard Young tableau for any $d\geq 3$. \\
    (\textsc{b}) A depiction of $T$ augmented with a shifted copy of the first row above row $d=3$. The value of $\MCW_3(T)$ is the minimum shift needed to ensure that the entries in this augmented picture are strictly increasing along columns. It is also required that there are no gaps between cells in a column, so $\MCW_d(T)=6$ for all $d\geq 4$.}
\end{figure}

\subsection{Oscillating tableaux} \label{sec:oscillating}
Given a type sequence $w=w_1\cdots w_r$, we define a \emph{semistandard $w$-oscillating tableau} to be a finite sequence of partitions $\emptyset = \lambda^{(0)},\lambda^{(1)},\ldots,\lambda^{(r)} = \emptyset$ such that for all $1 \leq i \leq r$ if $w_i=+$ then $\lambda^{(i-1)} \prec \lambda^{(i)}$ and if $w_i=-$ then $\lambda^{(i-1)} \succ \lambda^{(i)}$. We will also assume $w$ is part of the data of such a tableau. That is, a sequence $T$ can actually be a semistandard $w$-oscillating tableau for multiple different type sequences $w$, but we will treat these as different objects.

A semistandard $w$-oscillating tableau is called a \emph{$d$-semistandard $w$-oscillating tableau} if all of its constituent partitions are $d$-partitions. Moreover, we call it a \emph{$(d,L)$-cylindric semistandard $w$-oscillating tableau} if we can replace all the interlacing relations in the definition with $(d,L)$-interlacing relations. Given a $d$-semistandard oscillating tableau $T=(\lambda^{(0)},\ldots,\lambda^{(r)})$, we define its minimum cylindric width, denoted $\MCW_d(T)$, to be the smallest $L$ such that $T$ is $(d,L)$-cylindric. Equivalently, $\MCW_d(T) = \max_{1 \leq i \leq r} \MCW_d(\lambda^{(i-1)},\lambda^{(i)})$.

The analogue of \emph{weight} in the oscillating setting is defined as follows. Given a semistandard $w$-oscillating tableau $T=(\emptyset = \lambda^{(0)},\lambda^{(1)},\ldots,\lambda^{(r)} = \emptyset)$, we associate two vectors of nonnegative integers, $\wt^+(T)$ and $\wt^-(T)$, to it. The $i$th entry of $\wt^+(T)$ is $\wt^+(T)_i = |\lambda^{(a_i)}| - |\lambda^{(a_i-1)}|$ where $a_i$ is the index of the $i$th occurrence of `$+$' in $w$. The $i$th entry of $\wt^-(T)$ is $\wt^-(T)_i = |\lambda^{(b_i-1)}| - |\lambda^{(b_i)}|$ where $b_i$ is the index of the $i$th-to-last occurrence of `$-$' in $w$. (Note the difference in the ordering between the two definitions.)

The definitions of $\wt^+(T)$ and $\wt^-(T)$ are chosen with the following special case in mind. Suppose that $w$ is the type sequence consisting of $n$ `$+$'s followed by $m$ `$-$'s. Then a semistandard $w$-oscillating tableau is a sequence $T=(\emptyset = \lambda^{(0)} \prec \cdots \prec \lambda^{(n)} \succ \cdots \succ \lambda^{(n+m)} = \emptyset)$. This tableau $T$ can be split into two semistandard Young tableaux $P$ and $Q$ of shape $\lambda^{(n)}$ that are made up of the first $n$ and last $m$ partitions in the sequence, respectively. By the definitions of weight, we have $\wt(P) = \wt^+(T)$ and $\wt(Q) = \wt^-(T)$.

\subsection{Standard tableaux}
For each of the different kinds of tableaux defined above, we replace the word ``semistandard'' with ``standard'' in the special case where the associated weight vectors consist of all $1$'s. For example, a \emph{standard Young tableau} is a semistandard Young tableau $T$ such that $\wt(T) = (1,1,\ldots,1)$. A \emph{standard $w$-oscillating tableau} is a semistandard $w$-oscillating tableau $T$ such that $\wt^+(T) = (1,1,\ldots,1)$ and $\wt^-(T) = (1,1,\ldots,1)$.

\section{Main results} \label{sec:mainresults}

\subsection{Growth diagrams and local rules}

In order to state our main results we must also define the notion of a \emph{growth diagram}. For our purposes, a growth diagram consists of the following data.
\begin{enumerate}[(i)]
    \item A Young diagram $F$, along with a filling of $F$.
    \item An assignment of a partition to every lattice point of $F$. 
\end{enumerate}

This data is required to satisfy certain constraints depending on the kind of growth diagram under consideration. In this section we will consider \emph{RSK growth diagrams} and \emph{$d$-RSK growth diagrams}. For both of these, we impose a boundary condition: all of the partitions assigned to lattice points on the coordinate axes must be the empty partition.

The individual cells of a growth diagram are always required to satisfy a \emph{local rule}. In an RSK growth diagram, every cell must satisfy the \emph{RSK local rule}, and in a $d$-RSK growth diagram, every cell must satisfy the \emph{$d$-RSK local rule}.

A local rule specifies constraints on the data associated to a single cell.\footnote{Some authors use the term ``local rule'' to refer to a function or algorithm for constructing a single cell in a growth diagram. For our purposes it is more convenient to use a constraint-based approach.}  More precisely, any cell has an entry $m$ and four partitions $\kappa, \mu, \nu, \rho$ assigned to the corners as depicted in Figure \ref{fig:single_cell}. The local rule specifies the permitted values for these five pieces of data. The relevant rules for RSK and $d$-RSK growth diagrams are as follows.

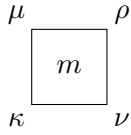
\begin{figure}
\centering
\begin{tikzpicture}
    \draw (0,0) rectangle (1,1);
    \node at (.5,.5) {$m$};
    \node at (-.2,-.2) {$\kappa$};
    \node at (-.2,1.2) {$\mu$};
    \node at (1.2,-.2) {$\nu$};
    \node at (1.2,1.2) {$\rho$};
\end{tikzpicture}
\caption{A single cell in a growth diagram.}
\label{fig:single_cell}
\end{figure}

\vspace{1em}
\textbf{RSK local rule:} It is required that $\mu \succ \kappa \prec \nu$ and $\mu \prec \rho \succ \nu$. Moreover,
\begin{equation} \label{eq:RSK}
    \begin{pmatrix}
        \rho_1 \\ \rho_2 \\ \rho_3 \\ \vdots
    \end{pmatrix}
    +
    \begin{pmatrix}
        0 \\ \kappa_1 \\ \kappa_2 \\ \vdots
    \end{pmatrix}=
    \begin{pmatrix}
        m + \max(\mu_1,\nu_1)\\ \min(\mu_1,\nu_1) + \max(\mu_2,\nu_2) \\ \min(\mu_2,\nu_2) + \max(\mu_3,\nu_3) \\ \vdots
    \end{pmatrix}.
\end{equation}

\vspace{1.5em}

\textbf{$d$-RSK local rule:} It is required that $\mu \succ \kappa \prec \nu$ and $\mu \prec \rho \succ \nu$. Also, all four partitions are $d$-partitions, and it is required that $m=0$ or $\kappa_d=0$. Moreover,
\begin{equation} \label{eq:dRSK}
    \begin{pmatrix}
        \rho_1 \\ \rho_2 \\ \rho_3 \\ \vdots \\ \rho_d
    \end{pmatrix}
    +
    \begin{pmatrix}
        \kappa_d \\ \kappa_1 \\ \kappa_2 \\ \vdots \\ \kappa_{d-1}
    \end{pmatrix}=
    \begin{pmatrix}
        m + \min(\mu_d,\nu_d) + \max(\mu_1,\nu_1)\\ \min(\mu_1,\nu_1) + \max(\mu_2,\nu_2) \\ \min(\mu_2,\nu_2) + \max(\mu_3,\nu_3) \\ \vdots \\ \min(\mu_{d-1},\nu_{d-1}) + \max(\mu_d,\nu_d)
    \end{pmatrix}.
\end{equation}
\vspace{1.5em}

The RSK local rule is well known (cf. \cite[Sec. 4.1]{krattenthaler2006}, \cite[Sec. 3.1]{vanleeuwen2005}). Growth diagrams satisfying this local rule encode the RSK correspondence in a natural way as we will recall shortly. The $d$-RSK local rule is new, but from the definition one can see that it is a ``cyclic'' modification of the RSK local rule. We will show that $d$-RSK growth diagrams encode a bijective correspondence that we call the \emph{$d$-RSK correspondence}. An example of a $d$-RSK growth diagram is shown in Figure~\ref{fig:dRSKexample}.

An important feature of the $d$-RSK local rule is that it reduces to the RSK local rule in certain cases. In particular, if $\min(\mu_d,\nu_d)=0$ then the two rules become equivalent. This will play an important role in our analysis. Any $d$-RSK growth diagram splits into two regions: the ``RSK region'' close to the axes where the cells satisfy $\min(\mu_d,\nu_d)=0$, and the ``cyclic region'' far from the axes where the cells satisfy $\min(\mu_d,\nu_d)>0$. The filling of the diagram will be almost entirely supported in the RSK region. This is because if $\min(\mu_d,\nu_d)>0$ then it is usually the case that $\kappa_d>0$ as well, in which case the $d$-RSK local rule requires that $m=0$. The example in Figure~\ref{fig:dRSKexample} illustrates this point.

\begin{figure}
\begin{tikzpicture}[
    xscale=1.5, yscale=1.5,
    partition/.style={fill=white, inner sep=3.0pt, text height=1.5ex, text depth=0.1ex},
    cross/.style={font=\LARGE, anchor=center}
]
    % Gray backgrounds
    \fill[gray!15] (0,0) rectangle (3,7);
    \fill[gray!15] (3,0) rectangle (4,6);
    \fill[gray!15] (4,0) rectangle (5,5);
    \fill[gray!15] (5,0) rectangle (7,3);
    % Grid
    \draw[step=1, gray!50, thin] (0,0) grid (7,7);
    % Filling
    \node[cross] at (0.5, 3.5) {$1$};
    \node[cross] at (0.5, 5.5) {$1$};
    \node[cross] at (0.5, 6.5) {$1$};
    \node[cross] at (1.5, 2.5) {$1$};
    \node[cross] at (1.5, 4.5) {$2$};
    \node[cross] at (2.5, 2.5) {$1$};
    \node[cross] at (2.5, 6.5) {$2$};
    \node[cross] at (3.5, 3.5) {$1$};
    \node[cross] at (3.5, 4.5) {$1$};
    \node[cross] at (4.5, 0.5) {$1$};
    \node[cross] at (4.5, 1.5) {$1$};
    \node[cross] at (4.5, 3.5) {$1$};
    \node[cross] at (4.5, 5.5) {$2$};
    \node[cross] at (5.5, 0.5) {$1$};
    \node[cross] at (5.5, 1.5) {$3$};
    \node[cross] at (5.5, 2.5) {$1$};
    \node[cross] at (5.5, 3.5) {$1$};
    \node[cross] at (6.5, 0.5) {$1$};
    % Partitions
    \node[partition] at (0,0) {$\emptyset$};
    \node[partition] at (0,1) {$\emptyset$};
    \node[partition] at (0,2) {$\emptyset$};
    \node[partition] at (0,3) {$\emptyset$};
    \node[partition] at (0,4) {$\emptyset$};
    \node[partition] at (0,5) {$\emptyset$};
    \node[partition] at (0,6) {$\emptyset$};
    \node[partition] at (0,7) {$\emptyset$};
    \node[partition] at (1,0) {$\emptyset$};
    \node[partition] at (1,1) {$\emptyset$};
    \node[partition] at (1,2) {$\emptyset$};
    \node[partition] at (1,3) {$\emptyset$};
    \node[partition] at (1,4) {$1$};
    \node[partition] at (1,5) {$1$};
    \node[partition] at (1,6) {$2$};
    \node[partition] at (1,7) {$3$};
    \node[partition] at (2,0) {$\emptyset$};
    \node[partition] at (2,1) {$\emptyset$};
    \node[partition] at (2,2) {$\emptyset$};
    \node[partition] at (2,3) {$1$};
    \node[partition] at (2,4) {$11$};
    \node[partition] at (2,5) {$31$};
    \node[partition] at (2,6) {$32$};
    \node[partition] at (2,7) {$33$};
    \node[partition] at (3,0) {$\emptyset$};
    \node[partition] at (3,1) {$\emptyset$};
    \node[partition] at (3,2) {$\emptyset$};
    \node[partition] at (3,3) {$2$};
    \node[partition] at (3,4) {$21$};
    \node[partition] at (3,5) {$32$};
    \node[partition] at (3,6) {$321$};
    \node[partition] at (3,7) {$531$};
    \node[partition] at (4,0) {$\emptyset$};
    \node[partition] at (4,1) {$\emptyset$};
    \node[partition] at (4,2) {$\emptyset$};
    \node[partition] at (4,3) {$2$};
    \node[partition] at (4,4) {$31$};
    \node[partition] at (4,5) {$43$};
    \node[partition] at (4,6) {$431$};
    \node[partition] at (4,7) {$542$};
    \node[partition] at (5,0) {$\emptyset$};
    \node[partition] at (5,1) {$1$};
    \node[partition] at (5,2) {$2$};
    \node[partition] at (5,3) {$22$};
    \node[partition] at (5,4) {$421$};
    \node[partition] at (5,5) {$442$};
    \node[partition] at (5,6) {$742$};
    \node[partition] at (5,7) {$853$};
    \node[partition] at (6,0) {$\emptyset$};
    \node[partition] at (6,1) {$2$};
    \node[partition] at (6,2) {$51$};
    \node[partition] at (6,3) {$621$};
    \node[partition] at (6,4) {$841$};
    \node[partition] at (6,5) {$844$};
    \node[partition] at (6,6) {$874$};
    \node[partition] at (6,7) {$985$};
    \node[partition] at (7,0) {$\emptyset$};
    \node[partition] at (7,1) {$3$};
    \node[partition] at (7,2) {$52$};
    \node[partition] at (7,3) {$622$};
    \node[partition] at (7,4) {$842$};
    \node[partition] at (7,5) {$944$};
    \node[partition] at (7,6) {$974$};
    \node[partition] at (7,7) {$995$};
\end{tikzpicture}
\caption{A $d$-RSK growth diagram on a $7\times 7$ square Young diagram, with $d=3$. For display purposes, cells whose entry is 0 are left empty. Partitions are written as concatenated lists of their parts. All cells where $\min(\mu_d,\nu_d)=0$ are shaded gray. These cells satisfy the RSK local rule.}
\label{fig:dRSKexample}
\end{figure}
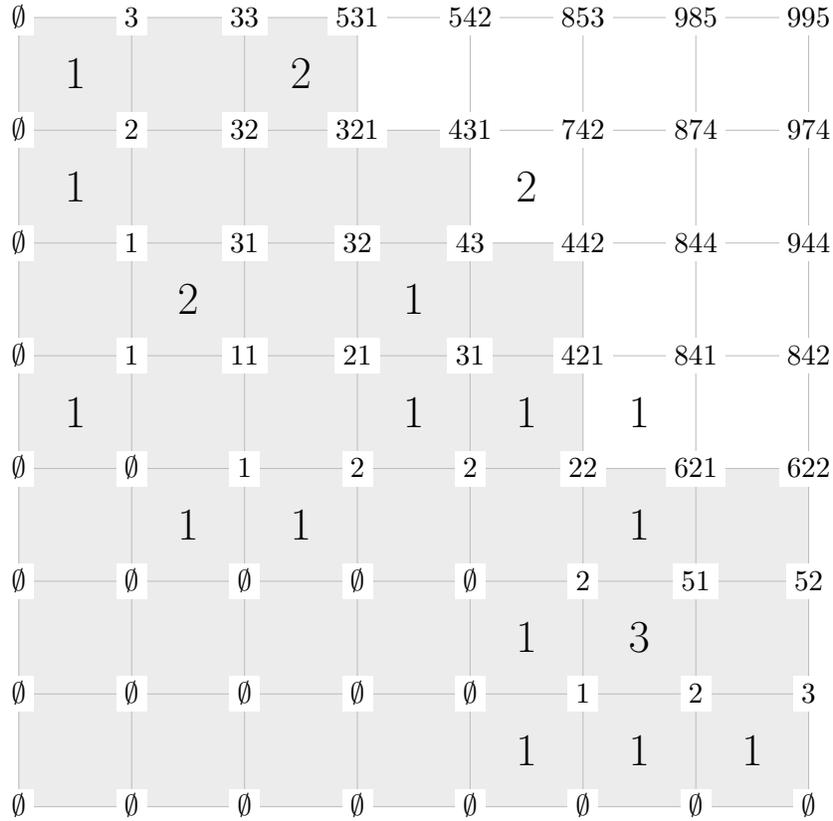

Before we discuss the RSK and $d$-RSK correspondences, we list some basic facts about both types of growth diagram. 

\begin{proposition} \label{prop:dRSKadjacent}
    Fix a $d$-RSK or RSK growth diagram on a Young diagram $F$. Let $\alpha$ and $\beta$ be two partitions assigned to lattice points in the diagram. The following properties hold.
    \begin{enumerate}[(i)]
        \item If $\alpha$ is located weakly below and weakly to the left of $\beta$ in the diagram, then $\alpha \subseteq \beta$.
        \item Suppose $\alpha$ and $\beta$ are assigned to lattice points connected by an edge. If the edge is horizontal and $\alpha$ is to the left of $\beta$, then $|\beta| - |\alpha|$ equals the sum of the entries in the column below the edge. If the edge is vertical and $\alpha$ is below $\beta$, then $|\beta| - |\alpha|$ equals the sum of the entries in the row to the left of that edge.
        \item Reflecting all of the data of the growth diagram across the line $y=x$ gives another growth diagram of the same type (i.e. RSK or $d$-RSK) on the reflected Young diagram $F'$.
    \end{enumerate} 
\end{proposition}

\begin{proof}
    The proofs of these properties are the same for either type of growth diagram. For property (i), if $\alpha$ and $\beta$ are connected by an edge then the local rule implies $\alpha \prec \beta$, which is stronger than $\alpha \subseteq \beta$. The general case follows by transitivity of the $\subseteq$ relation. For property (ii), one may note that the local rule implies $|\rho|+|\kappa|=m+|\mu|+|\nu|$. Inductively applying this relation along any fixed row or column proves the statement. Property (iii) follows from the symmetry of the local rule with respect to swapping $\mu$ and $\nu$.
\end{proof}

\subsection{Growth diagram correspondences}

We now recall how RSK growth diagrams lead to the RSK correspondence. This construction goes back to work of Fomin \cite{fomin1986} and Roby \cite{roby1991}. The main idea is that any RSK growth diagram can be recovered from just a small portion of the data. A precise statement is given as follows.

\begin{theorem}[Oscillating RSK correspondence; see {\cite[Thm. 7]{krattenthaler2006}}]\label{thm:RSKgrowth}
    Let $F$ be a Young diagram and let $w\in \{+,-\}^k$ be the type sequence encoding the outer boundary of $F$. 
    
    \begin{enumerate}[(a)]
        \item For any RSK growth diagram on $F$, the sequence $\lambda^{(0)},\lambda^{(1)},\ldots,\lambda^{(k)}$ of partitions assigned to the outer boundary of $F$ (listed in order starting from the point on the positive $x$-axis) forms a semistandard $w$-oscillating tableau.
        \item Any filling of $F$ can be extended uniquely to an RSK growth diagram on $F$.
        \item Any semistandard $w$-oscillating tableau assigned to the lattice points along the outer boundary of $F$ can be uniquely extended to an RSK growth diagram on $F$.
    \end{enumerate}
    Consequently, there is an explicit bijective correspondence between fillings of $F$ and semistandard $w$-oscillating tableaux.
\end{theorem}

The last sentence of this theorem follows by putting (a), (b), and (c) together. Indeed, (a) and (c) show that RSK growth diagrams on $F$ are in bijection with semistandard $w$-oscillating tableaux. Similarly, (b) shows that RSK growth diagrams on $F$ are in bijection with fillings of $F$. Composing these bijections gives the desired bijection between fillings and tableaux. 

We can think of Theorem \ref{thm:RSKgrowth} as providing a family of bijections (one for each choice of $F$). We can also view this as a single unified bijection whose domain and codomain are obtained by taking disjoint unions of the relevant sets of fillings and tableaux. We will refer to this unified bijection as the \emph{oscillating RSK correspondence}.

The classical RSK correspondence between matrices and pairs of semistandard Young tableaux can be recovered by restricting the oscillating correspondence to rectangular Young diagrams. When $F$ is rectangular, the type sequence encoding the outer boundary of $F$ is of the form $w=+^n-^m$. For any such $w$, a semistandard $w$-oscillating tableau can be identified with a pair of semistandard Young tableaux of the same shape. In terms of the growth diagram, this pair is obtained by reading off the partitions along the right and top sides of $F$ separately. The filling of $F$ is identified with a matrix in the usual way. The fact that this special case coincides with the classical RSK correspondence defined via Schensted insertion is well known (see \cite{krattenthaler2006}, \cite{roby1991}).

Our first main result is an analogue of Theorem \ref{thm:RSKgrowth} for $d$-RSK growth diagrams. The proof will be given in Section \ref{sec:dRSKgrowth}.

\begin{theorem}[Oscillating $d$-RSK correspondence]\label{thm:dRSKgrowth}
    Let $F$ be a Young diagram, and let $w\in\{+,-\}^k$ be the type sequence encoding the outer boundary of $F$. Let $d\in \N$.
    \begin{enumerate}[(a)]
        \item For any $d$-RSK growth diagram on $F$, the sequence of partitions $\lambda^{(0)},\ldots, \lambda^{(k)}$ on the outer boundary of $F$ (listed in order starting from the point on the positive $x$-axis) forms a $d$-semistandard $w$-oscillating tableau. Also, the filling of the diagram avoids the pattern $d\cdots 1 (d+1)$. 
        \item Any filling of $F$ that avoids the pattern $d\cdots 1 (d+1)$ extends uniquely to a $d$-RSK growth diagram on $F$.
        \item Any $d$-semistandard $w$-oscillating tableau assigned to the lattice points along the outer boundary of $F$ can be uniquely extended to a $d$-RSK growth diagram on $F$.
    \end{enumerate}
    Consequently, there is an explicit bijective correspondence between fillings of $F$ that avoid the pattern $d\cdots 1 (d+1)$ and $d$-semistandard $w$-oscillating tableaux.
\end{theorem}

Part (a) of this theorem differs from the RSK case in an essential way. The partitions on the outer boundary must now form a $d$-semistandard $w$-oscillating tableau (this is immediate from the local rule), and there is now a constraint on the \emph{filling} of the diagram. Part (b) shows that this filling constraint is exactly the ``correct'' one. Indeed, (a) and (b) together show that $d$-RSK growth diagrams on $F$ are in bijection with fillings of $F$ that avoid the pattern $d\cdots 1 (d+1)$. Similarly, (a) and (c) imply that $d$-RSK growth diagrams on $F$ are in bijection with $d$-semistandard $w$-oscillating tableaux. 

We will call the bijection between fillings and tableaux given by Theorem \ref{thm:dRSKgrowth} the \emph{oscillating $d$-RSK correspondence}. As with RSK, we treat this as a unified bijection over all choices of $F$. Restricting to rectangular Young diagrams gives a non-oscillating version of the $d$-RSK correspondence, which is a bijection between fillings of rectangular Young diagrams that avoid the pattern $d\cdots 1 (d+1)$ and pairs of $d$-semistandard Young tableaux with the same shape. We refer to this as the \emph{$d$-RSK correspondence}.

The growth diagram shown in Figure \ref{fig:dRSKexample} illustrates an example of the (oscillating) $d$-RSK correspondence for $d=3$. The filling of the diagram avoids the pattern $3214$. The sequence of partitions along the outer boundary forms a $3$-semistandard $w$-oscillating tableau where $w=+^7-^7$. Alternatively, by reading off the partitions along the right and top sides of the diagram separately, this is a pair of $3$-semistandard Young tableaux of shape $(9,9,5)$.

\subsection{Properties of the oscillating RSK and $d$-RSK correspondences}
From the growth diagram constructions and Proposition \ref{prop:dRSKadjacent}, we can deduce some properties that the oscillating RSK and $d$-RSK correspondences have in common.

\begin{proposition} \label{prop:weightreflection}
Let $F$ be a Young diagram whose outer boundary is encoded by the type sequence $w$. Let $T$ be a $d$-semistandard $w$-oscillating tableau. Under the $d$-RSK correspondence there is a filling of $F$ that corresponds to $T$. The following properties hold.
\begin{enumerate}
    \item The sum of the entries in the $i$th row of the filling equals $\wt^+(T)_i$. The sum of the entries in the $j$th column of the filling equals $\wt^-(T)_j$.
    \item Reversing the sequence $T$ corresponds to reflecting the filling across the line $y=x$.
\end{enumerate}
These properties also hold for the oscillating RSK correspondence.
\end{proposition}
\begin{proof}
    The first statement follows from Proposition \ref{prop:dRSKadjacent}(ii). The second statement follows from Proposition \ref{prop:dRSKadjacent}(iii).
\end{proof}

The reflection property stated in this proposition merits some further discussion. The reversal of $T$ is to be interpreted as a $d$-semistandard $w'$-oscillating tableau where $w'$ is the type sequence obtained from $w$ by reversing it and switching `$+$'s and `$-$'s. The reflection of a filling of $F$ across $y=x$ is to be interpreted as a filling of the reflected Young diagram $F'$ (whose outer boundary is encoded by $w'$).

To refer to the fact that these two reflection operations correspond to each other under the oscillating RSK or $d$-RSK correspondence, we say that the correspondences \emph{commute with reflection}. Properly interpreting this statement requires viewing the correspondences as unified bijections over all choices of Young diagrams.

Alternatively, this commutativity property is also meaningful at the level of a single Young diagram $F$ if $F$ happens to be symmetric about $y=x$. In this special case, this property has an interesting consequence: symmetric fillings of $F$ correspond to ``palindromic'' tableaux (i.e. those that are invariant under reversal).

The first property stated in Proposition \ref{prop:weightreflection} is more straightforward. We will refer to this as the \emph{weight-preserving} property of RSK and $d$-RSK. Although we did not specifically define the notion of weight for fillings, it should be clear why we have chosen this name.

Both properties in Proposition \ref{prop:weightreflection} become somewhat simpler to state for the non-oscillating RSK and $d$-RSK correspondences. In that case, the filling of $F$ corresponds to a pair $(P,Q)$ of semistandard Young tableaux. The weight-preserving property says that the row and column sums of the filling give $\wt(P)$ and $\wt(Q)$. The fact that RSK and $d$-RSK commute with reflection says that reflecting the filling across $y=x$ (i.e. transposing the matrix) corresponds to interchanging $P$ and $Q$.

\subsection{Chains in growth diagrams}

We now turn to our other main result about $d$-RSK. First, we state the analogous result for RSK, which is well known.

\begin{theorem}[see {\cite[Thm. 8]{krattenthaler2006}}] \label{thm:RSKchains}
Fix an RSK growth diagram on a Young diagram $F$. Let $p$ be a lattice point of $F$, and let $\lambda$ be the partition assigned to $p$ in the growth diagram. The length of the longest se-chain in $\Rect_p$ is equal to $\ell(\lambda)$. The length of the longest NE-chain in $\Rect_p$ is equal to $\lambda_1$. 
\end{theorem}

This theorem is a growth diagram analogue of a classical theorem of Schensted \cite[Thm.~1]{schensted1961}. It has a generalization due to Greene \cite{greene1974} (see \cite[Thm. 8]{krattenthaler2006}) that gives a complete description of the partitions in the growth diagram in terms of chains, but we will not need this generalization.

The following theorem is an analogue of Theorem \ref{thm:RSKchains} for $d$-RSK growth diagrams. The proof will be given in Section \ref{sec:dRSKchains}.

\begin{theorem}\label{thm:dRSKchains}
    Fix a $d$-RSK growth diagram on a Young diagram $F$. Let $p$ be a lattice point of $F$, and let $\lambda$ be the partition assigned to $p$ in the growth diagram. Then $\ell(\lambda) = \min(d, K)$ where $K$ is the length of the longest se-chain contained in $\Rect_p$.
    
    Let $T$ be the $d$-semistandard $w$-oscillating tableau obtained from the partitions on the outer boundary of $F$. The length of the longest NE-chain in $F$ equals the minimum cylindric width of $T$.
\end{theorem}

The statement about se-chains in this theorem is consistent with our previous discussion of how the $d$-RSK local rule reduces to the RSK local rule in certain cases. The statement about NE-chains is more interesting since it reveals a connection to cylindric tableaux. 

As an example, let us consider Figure \ref{fig:dRSKexample}. The length of the longest NE-chain in this diagram is $1+1+3+1+1=7$. The minimum cylindric width of the tableau on the outer boundary is achieved by the pair of adjacent partitions $(9,4,4)$ and $(8,4,2)$ on the right side. That is, $\MCW_3((8,4,2),(9,4,4))=9-2=7$, and this is the largest value of $\MCW_3$ among all pairs of adjacent partitions on the outer boundary. As a consequence, the tableau on the outer boundary is $(3,L)$-cylindric if and only if $L\geq 7$.

In the case of rectangular Young diagrams, there is an interesting contrast between the RSK and $d$-RSK theorems. In the case of RSK, the length of the longest NE-chain can be determined by looking at just the partition in top-right corner of the rectangle. On the other hand, for $d$-RSK the minimum cylindric width statistic may depend on other partitions along the outer boundary. The example in Figure \ref{fig:dRSKexample} illustrates this point.

It should be noted that Theorem \ref{thm:dRSKchains} can be restated to give information about the longest NE-chain in any sub-Young diagram $G \subseteq F$. Indeed, we can always restrict a growth diagram to a smaller growth diagram, so the length of the longest NE-chain in $G$ is equal to the minimum cylindric width of the sequence of partitions along the outer boundary of $G$. 

By the same token, Theorem \ref{thm:RSKchains} can also be restated to give information about NE-chains in any sub-Young diagram $G \subseteq F$, not just rectangular ones. In that case the length of the longest NE-chain in $G$ is equal to $\max(\lambda_1)$ where $\lambda$ ranges over all partitions along the outer boundary of $G$. Hence, the two theorems essentially hold at the same level of generality.

\section{Cylindric correspondences} \label{sec:corollaries}

Our main results on $d$-RSK immediately imply the following correspondence involving cylindric tableaux.

\begin{corollary}[Oscillating cylindric RSK correspondence]
    Let $(d,L)$ be a pair of positive integers. Let $F$ be a Young diagram and let $w$ denote the type sequence encoding the outer boundary of $F$.  The oscillating $d$-RSK correspondence restricts to a bijection between
    \begin{enumerate}
        \item fillings of $F$ that avoid the pattern $d\cdots 1 (d+1)$ and do not contain a NE-chain of length greater than $L$, and
        \item $(d,L)$-cylindric semistandard $w$-oscillating tableaux.
    \end{enumerate}
\end{corollary}
\begin{proof}
    This correspondence is obtained by restricting the oscillating $d$-RSK correspondence. By Theorem \ref{thm:dRSKchains}, restricting to fillings that do not contain a NE-chain of length greater than $L$ corresponds to restricting to tableaux that are $(d,L)$-cylindric.
\end{proof}

We call the above correspondence the \emph{oscillating $(d,L)$-cylindric RSK correspondence}. As usual we think of this as a unified bijection over all choices of $F$. This correspondence inherits the properties of the oscillating $d$-RSK correspondence stated in Proposition \ref{prop:weightreflection} (that is, it is weight-preserving and commutes with reflection). 

When talking about this cylindric correspondence, we will omit the ``$(d,L)$-'' from the name since the relevant parameters will always be clear from context. The same convention will apply to the non-oscillating version, which we state now.

\begin{corollary}[Cylindric RSK correspondence] 
    Let $(d,L)$ be a pair of positive integers. The $d$-RSK correspondence restricts to a bijection between
    \begin{enumerate}
        \item fillings of rectangular Young diagrams that avoid the pattern $d\cdots 1 (d+1)$ and do not contain a NE-chain of length greater than $L$, and
        \item pairs $(P,Q)$ where $P=(\lambda^{(0)}, \lambda^{(1)}, \ldots, \lambda^{(n)})$ and $Q=(\mu^{(0)}, \mu^{(1)}, \ldots, \mu^{(m)})$ are $(d,L)$-cylindric semistandard Young tableaux with a common shape $\lambda^{(n)} = \mu^{(m)}$.
    \end{enumerate}
    Under this bijection, $n$ and $m$ equal the number of rows and columns of the rectangular Young diagram, respectively. The sum of entries in the $i$th row of the filling equals $\wt(P)_i$, and the sum of entries in the $j$th column of the filling equals $\wt(Q)_j$. Reflecting the filling across the line $y=x$ corresponds to swapping $P$ and $Q$ under this bijection.
\end{corollary}
\begin{proof}
    This is obtained by restricting the oscillating cylindric RSK correspondence to rectangular Young diagrams. In terms of $d$-RSK growth diagrams, the $P$ and $Q$ tableaux are obtained by reading off the partitions along the right and top sides of the rectangular Young diagram, respectively.
\end{proof}

The cylindric RSK correspondence immediately implies the cylindric RS correspondence stated in the introduction (i.e. Theorem~\ref{thm:cRS}). Indeed, the cylindric RS correspondence is the special case where the row and column sums are all 1. In that case the fillings can be identified with permutation matrices and hence permutations.

The cylindric RS correspondence yields the following enumerative consequence.

\begin{corollary}
    Let $d,L,n$ be positive integers. The number of $(d,L)$-cylindric standard Young tableaux of size $n$ is equal to the number of involutions in $S_n$ that avoid the patterns $d \cdots 1 (d+1)$ and $1 \cdots (L+1)$.
\end{corollary}
\begin{proof}
    Under the cylindric RS correspondence, involutions in $S_n$ that avoid the patterns $d \cdots 1 (d+1)$ and $1 \cdots (L+1)$ correspond to pairs $(T,T)$ where $T$ is a $(d,L)$-cylindric standard Young tableau of size $n$.
\end{proof}

For comparison, the number of standard Young tableaux of size $n$ is equal to the number of involutions in $S_n$. The number of $d$-standard Young tableaux of size $n$ is equal to the number of involutions in $S_n$ that avoid the pattern $(d+1) \cdots 1$. These facts follow from the classical RS correspondence. In Section \ref{sec:permutations} we will discuss other sets of permutations that are also in bijection with these.

\begin{remark}
    Recently, Huh, Kim, Krattenthaler, and Okada \cite{huh2025kkc} discovered other combinatorial expressions for the number of $(d,L)$-cylindric standard Young tableaux of a given size. It would be interesting to establish a direct bijection between their expressions and the one given above. The main results of Huh et al. are certain symmetric function identities, which are affine analogues of bounded Littlewood identities (see \cite[Sec. 1]{huh2025kkc}). Even in the non-affine setting these bounded Littlewood identities are not yet well understood from a bijective perspective (see \cite{fischer2024}).
\end{remark}

\section{Proof of Theorem \ref{thm:dRSKchains}} \label{sec:dRSKchains}

We now prove our results on se-chains and NE-chains in $d$-RSK growth diagrams. We give this proof prior to the proof of Theorem \ref{thm:dRSKgrowth} because the result on se-chains will be needed in the proof of Theorem \ref{thm:dRSKgrowth} (a) and (b).

\subsection{Lemmas}
We first prove two lemmas that give sufficient conditions for deducing that certain cells in a $d$-RSK growth diagram satisfy the ordinary RSK local rule.

\begin{lemma} \label{lem:dRSKchains1}
Fix a $d$-RSK growth diagram on a Young diagram $F$. Let $p=(x,y)$ be a lattice point of $F$. If the partition assigned to $p$ has length less than $d$, then all the cells in $\Rect_p$ satisfy the RSK local rule. More generally, all the cells of $\Rect_{(x+1,y)}$ and $\Rect_{(x,y+1)}$ do as well (assuming these cells are in $F$).
\end{lemma}
\begin{proof}
First we restate an observation made previously: if a cell satisfies the $d$-RSK local rule and either of the partitions on the top-left or bottom-right corner has length less than $d$, then the cell also satisfies the RSK local rule. This is a straightforward consequence of the definitions of the local rules. 

This observation is sufficient to see that the lemma holds. Indeed, the hypothesis states that the partition assigned to $p$ has length less than $d$, and this implies that the partition at every lattice point of $\Rect_p$ has length less than $d$ as well. The lemma follows by applying the observation to every cell in $\Rect_p$ and to the cells adjacent to $\Rect_p$.
\end{proof}

\begin{lemma} \label{lem:dRSKchains2} Fix a $d$-RSK growth diagram on a Young diagram $F$. Let $p$ be a lattice point of $F$ such that $\Rect_p$ contains no se-chain of length $d$. Then all the cells in $\Rect_p$ satisfy the RSK local rule.
\end{lemma}
\begin{proof}
We prove the statement by induction. Fix a choice of growth diagram, and let ``$P(x,y)$'' denote the statement of the lemma at the point $p=(x,y)$. It is clear that $P(0,0)$ is true. Now suppose $P(x,y)$ holds for some arbitrary $(x,y)$. We claim that $P(x+1,y)$ also holds. Indeed, if $\Rect_{(x+1,y)}$ contains no se-chain of length $d$, then neither does $\Rect_{(x,y)}$, so by the inductive hypothesis we know that all the cells in $\Rect_{(x,y)}$ satisfy the RSK local rule. This means we can treat $\Rect_{(x,y)}$ as an RSK growth diagram, so the partition assigned to $(x,y)$ has length less than $d$ by Theorem~\ref{thm:RSKchains}. Hence, by the previous lemma all the cells in $\Rect_{(x+1,y)}$ satisfy the RSK local rule. A similar argument shows that $P(x,y)$ implies $P(x,y+1)$. This completes the induction.
\end{proof}

\subsection{se-chains}
We can use the lemmas to quickly deduce the statement about se-chains in Theorem \ref{thm:dRSKchains}. Fix a point $p$ in a $d$-RSK growth diagram. Let $\lambda$ denote the partition assigned to $p$, and let $K$ denote the length of the longest se-chain in $\Rect_p$. If either $\ell(\lambda) < d$ or $K < d$, then Lemma \ref{lem:dRSKchains1} or Lemma \ref{lem:dRSKchains2} imply that all the cells in $\Rect_p$ satisfy the RSK local rule. Hence, Theorem \ref{thm:RSKchains} implies $\ell(\lambda) = K$, and so $\ell(\lambda) = \min(d,K)$. The only remaining case is when both $\ell(\lambda) = d$ and $K \geq d$. In that case, clearly $\ell(\lambda) = \min(d,K)$.

\subsection{NE-chains} \label{sec:NEchains}
Now we prove the statement about NE-chains in Theorem~\ref{thm:dRSKchains}. We will prove it by induction on the size of the Young diagram $F$. The base case when $F$ is empty is trivial. When $F$ is nonempty, we inductively assume the statement holds for all smaller diagrams.

Fix an arbitrary $d$-RSK growth diagram on $F$. Choose a cell $C$ with the property that removing $C$ from $F$ yields a smaller Young diagram $\widetilde{F}$. Our inductive hypothesis implies that the length of the longest NE-chain in $\widetilde{F}$ equals the minimum cylindric width of the $d$-semistandard oscillating tableau obtained from the outer boundary of $\widetilde{F}$.

Let $T=(\lambda^{(0)},\lambda^{(1)},\ldots,\lambda^{(r)})$ denote the $d$-semistandard oscillating tableau obtained from the outer boundary of $F$. There exists some index $j$ such that $\lambda^{(j-1)}, \lambda^{(j)}, \lambda^{(j+1)}$ are the partitions assigned to the bottom-right, top-right, and top-left corners of $C$, respectively. Let $\widetilde{\lambda}^{(j)}$ denote the partition assigned to the bottom-left corner of $C$, and let $\widetilde{T}$ denote the $d$-semistandard oscillating tableau obtained from $T$ by replacing $\lambda^{(i)}$ with $\widetilde{\lambda}^{(i)}$. Then $\widetilde{T}$ is precisely the oscillating tableau obtained from the outer boundary of $\widetilde{F}$ \emph{except} it may contain an extra copy of the empty partition at the beginning or end if $C$ is adjacent to the $x$-axis or $y$-axis. These extra copies of the empty partition do not affect the minimum cylindric width of $\widetilde{T}$. Consequently, the length of the longest NE-chain in $\widetilde{F}$ equals $\MCW_d(\widetilde{T})$. We will use this to deduce that the length of the longest NE-chain in $F$ equals $\MCW_d(T)$.

Let $m$ denote the entry in the cell $C$. From the definition of minimum cylindric width and from the $d$-RSK local rule at $C$, we have
\begin{align}
\MCW_d(T) &= \max\left(\max_{i\neq j,j+1} \MCW_d(\lambda^{(i-1)}, \lambda^{(i)}), \lambda^{(j)}_1-\lambda^{(j-1)}_d,\lambda^{(j)}_1-\lambda^{(j+1)}_d\right) \nonumber \\
&= \max\left(\max_{i\neq j,j+1} \MCW_d(\lambda^{(i-1)}, \lambda^{(i)}), \lambda^{(j)}_1-\min(\lambda^{(j-1)}_d,\lambda^{(j+1)}_d)\right) \nonumber \\
&= \max\left(\max_{i\neq j,j+1} \MCW_d(\lambda^{(i-1)}, \lambda^{(i)}), m+\max(\lambda^{(j-1)}_1, \lambda^{(j+1)}_1) - \widetilde{\lambda}^{(j)}_d\right) \nonumber \\
&= \max\left(\max_{i\neq j,j+1} \MCW_d(\lambda^{(i-1)}, \lambda^{(i)}), m+\lambda^{(j-1)}_1 - \widetilde{\lambda}^{(j)}_d, m+\lambda^{(j+1)}_1 - \widetilde{\lambda}^{(j)}_d\right). \label{eq:MCWT}
\end{align}
The $d$-RSK local rule at $C$ also implies that at least one of $m$ or $\widetilde{\lambda}^{(j)}_d$ is zero. We will consider these two cases separately.

\textbf{Case 1: $m=0$.} In this case the length of the longest NE-chain in $F$ equals the length of the longest NE-chain in $\widetilde{F}$. Hence, it suffices to show that $\MCW_d(T) = \MCW_d(\widetilde{T})$. This follows immediately from \eqref{eq:MCWT} and the definition of $\MCW_d(\widetilde{T})$.

\textbf{Case 2: $\widetilde{\lambda}^{(j)}_d=0$.} First we introduce some additional notation. For any Young diagram $G \subseteq F$, we shall write $\LNE(G)$ to denote the length of the longest NE-chain in $G$ (with respect to the filling of $F$ that we have already fixed). Let $p=(x,y)$ denote the bottom-left corner of $C$. Then by the definition of NE-chains we have
\[
\LNE(F) = \max\left(\LNE(\widetilde{F}), m+\LNE(\Rect_{(x+1,y)}), m+\LNE(\Rect_{(x,y+1)})\right).
\]

We will rewrite the right-hand side of this expression. As noted previously, the induction hypothesis implies $\LNE(\widetilde{F}) = \MCW_d(\widetilde{T})$. We also claim that $\LNE(\Rect_{(x+1,y)}) = \lambda^{(j-1)}_1$ and $\LNE(\Rect_{(x,y+1)}) = \lambda^{(j+1)}_1$. To see this, first note that the partitions assigned to $(x+1,y)$ and $(x,y+1)$ are $\lambda^{(j-1)}$ and $\lambda^{(j+1)}$, respectively. Hence, the claim would follow from Theorem \ref{thm:RSKchains} if all the cells in $\Rect_{(x+1,y)}$ and $\Rect_{(x,y+1)}$ satisfy the RSK local rule. This holds by Lemma \ref{lem:dRSKchains1} since the partition assigned to $(x,y)$ is $\widetilde{\lambda}^{(j)}$, which is assumed to have length less than $d$ in this case.

Putting these results of the previous paragraph together we have
\begin{equation} \label{eq:LNEF}
\LNE(F) = \max\left(\MCW_d(\widetilde{T}), m+\lambda^{(j-1)}_1, m+\lambda^{(j+1)}_1\right).
\end{equation}
Also, from the definition of the minimum cylindric width of $\widetilde{T}$ we have
\begin{align*}
\MCW_d(\widetilde{T}) &= \max\left(\max_{i \neq j,j+1} \MCW_d(\lambda^{(i-1)}, \lambda^{(i)}),\lambda^{(j-1)}_1-\widetilde{\lambda}^{(j)}_d, \lambda^{(j+1)}_1 - \widetilde{\lambda}^{(j)}_d\right) \\
&= \max\left(\max_{i \neq j,j+1} \MCW_d(\lambda^{(i-1)}, \lambda^{(i)}),\lambda^{(j-1)}_1, \lambda^{(j+1)}_1\right).
\end{align*}
Inserting this into \eqref{eq:LNEF}, we obtain
\[
\LNE(F) = \max\left(\max_{i\neq j,j+1} \MCW_d(\lambda^{(i-1)}, \lambda^{(i)}), m+\lambda_1^{(j-1)}, m+\lambda_1^{(j+1)} \right).
\]
This is equal to $\MCW_d(T)$ by \eqref{eq:MCWT}, completing the proof. 

\section{Proof of Theorem \ref{thm:dRSKgrowth}} \label{sec:dRSKgrowth}

Next we prove our main result establishing the oscillating $d$-RSK correspondence. This is similar to standard proofs involving growth diagrams, but the $d$-RSK local rule introduces some complications that require some care.

\subsection{Local growth lemmas}
To construct $d$-RSK growth diagrams we first describe how to construct a single cell satisfying the $d$-RSK local rule. We will refer to the process of constructing a single cell as \emph{local growth}.\footnote{As we have noted previously, some authors use the term ``local rule'' for this.} There are two directions in which we can perform local growth. The first direction is illustrated in the following picture and is described in the subsequent lemma.

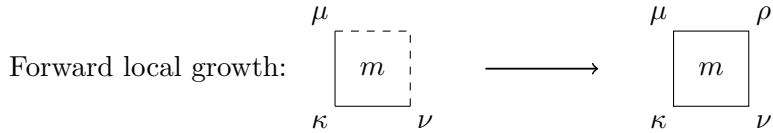
\begin{tikzpicture}
    \node[anchor=east] at (-3.25,0.5) {Forward local growth:};
    \begin{scope}[xshift=-2.75cm]
        \draw (0,0) -- (1,0);
        \draw (0,0) -- (0,1);
        \draw[dashed] (0,1) -- (1,1);
        \draw[dashed] (1,0) -- (1,1);
        \node at (.5,.5) {$m$};
        \node at (-.2,-.2) {$\kappa$};
        \node at (-.2,1.2) {$\mu$};
        \node at (1.2,-.2) {$\nu$};
    \end{scope}
    
    \draw[->,line width=.7pt] (-0.75,.5) -- (0.75,.5);
    
    \begin{scope}[xshift=1.75cm]
        \draw (0,0) rectangle (1,1);
        \node at (.5,.5) {$m$};
        \node at (-.2,-.2) {$\kappa$};
        \node at (-.2,1.2) {$\mu$};
        \node at (1.2,-.2) {$\nu$};
        \node at (1.2,1.2) {$\rho$};
    \end{scope}
\end{tikzpicture}

\begin{lemma}[Forward local growth]\label{lem:forwardlocal}
Suppose $\mu \succ \kappa \prec \nu$ are $d$-partitions and $m$ is a nonnegative integer. Also, suppose that $\kappa_d=0$ or $m=0$. Then there exists a unique $d$-partition $\rho$ such that the resulting cell with entry $m$ and corner partitions $\kappa,\mu,\nu,\rho$ (as labelled above) satisfies the $d$-RSK local rule.
\end{lemma}

\begin{proof}
From equation \eqref{eq:dRSK} for the $d$-RSK local rule, the values of $\rho_1,\ldots, \rho_d$ are completely determined. It is straightforward to check that $\mu \prec \rho \succ \nu$ follows as a consequence of the hypothesis $\mu \succ \kappa \prec \nu$. The fact that $\rho$ is a partition (i.e. $\rho_1 \geq \cdots \geq \rho_d$) follows \emph{a fortiori} from the interlacing inequalities.
\end{proof}

The other direction of local growth is illustrated in the following picture and is described in the subsequent lemma.

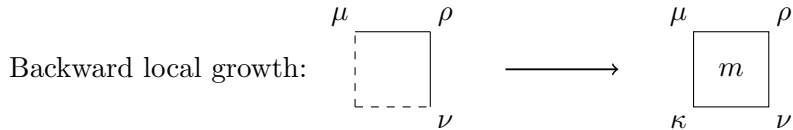
\begin{tikzpicture}
    \node[anchor=east] at (-3.25,0.5) {Backward local growth:};
    \begin{scope}[xshift=-2.75cm]
        \draw[dashed] (0,0) -- (1,0);
        \draw[dashed] (0,0) -- (0,1);
        \draw (0,1) -- (1,1);
        \draw (1,0) -- (1,1);
        \node at (-.2,1.2) {$\mu$};
        \node at (1.2,-.2) {$\nu$};
        \node at (1.2,1.2) {$\rho$};
    \end{scope}

    \draw[->,line width=.7pt] (-0.75,.5) -- (0.75,.5);

    \begin{scope}[xshift=1.75cm]
        \draw (0,0) rectangle (1,1);
        \node at (.5,.5) {$m$};
        \node at (-.2,-.2) {$\kappa$};
        \node at (-.2,1.2) {$\mu$};
        \node at (1.2,-.2) {$\nu$};
        \node at (1.2,1.2) {$\rho$};
    \end{scope}
\end{tikzpicture}

\begin{lemma}[Backward local growth]\label{lem:backwardlocal}
Suppose $\mu \prec \rho \succ \nu$ are $d$-partitions. Then there exists a unique $d$-partition $\kappa$ and a unique nonnegative integer $m$ such that the resulting cell with entry $m$ and corner partitions $\kappa,\mu,\nu,\rho$ (as labelled above) satisfies the $d$-RSK local rule.
\end{lemma}

\begin{proof}
The values of $\kappa_1,\ldots,\kappa_{d-1}$ are determined by equation \eqref{eq:dRSK}, and it is easy to check that the required interlacing inequalities hold for these parts. It remains to determine $\kappa_d$ and $m$. By picking out the first component of \eqref{eq:dRSK} and rearranging, we see that
\begin{equation}\label{eq:lem3} 
    \kappa_d - m = \min(\mu_d,\nu_d) + \max(\mu_1,\nu_1) - \rho_1.
\end{equation}
There may be many values of $\kappa_d,m\geq 0$ that satisfy this equation, but the $d$-RSK local rule requires that at least one of them must equal 0. Hence, if the right-hand side of \eqref{eq:lem3} is positive, we must take $m=0$ and define $\kappa_d$ accordingly. Similarly, if the right-hand side of \eqref{eq:lem3} is negative, we must take $\kappa_d=0$ and define $m$ accordingly. If the right-hand side is zero, then $m = \kappa_d=0$. Hence, the values of $\kappa_d$ and $m$ are uniquely determined, and it is easy to check $0\leq \kappa_d \leq \min(\mu_d,\nu_d)$. Consequently, the interlacing conditions $\mu \succ \kappa \prec \nu$ are satisfied, and $\kappa$ is indeed a partition.
\end{proof}

\subsection{Proof of theorem}

We are now in a position to prove Theorem \ref{thm:dRSKgrowth}. Throughout the remainder of the section, we fix a Young diagram $F$ whose outer boundary is encoded by the type sequence $w$.

\subsubsection{Proof of (a)}

Fix a $d$-RSK growth diagram on $F$. The fact that the sequence of partitions on the outer boundary forms a $d$-semistandard $w$-oscillating tableau is simply a consequence of the interlacing hypotheses in the $d$-RSK local rule.

It remains to show that the filling of the growth diagram avoids the pattern $d\cdots 1 (d+1)$. We will use the se-chain property in Theorem \ref{thm:dRSKchains}, which we proved in the previous section. Suppose for the sake of contradiction that the filling contains the pattern $d\cdots 1 (d+1)$. Fix a sequence of cells that witnesses this (i.e. the entries of these cells form an instance of the pattern). Let $C$ denote the final cell in the sequence and let $m$ denote the entry in $C$. Let $p$ denote the bottom-left corner of $C$, and let $\kappa$ denote the partition assigned to $p$. By the definition of what it means to contain the pattern $d\cdots 1 (d+1)$, we see that $m>0$ and $\Rect_p$ must contain a se-chain of length $d$. Hence, by Theorem \ref{thm:dRSKchains} we know that $\ell(\kappa)=d$. This means that $\kappa_d>0$, which is a contradiction since the $d$-RSK local rule for cell $C$ requires that either $m=0$ or $\kappa_d=0$.

\subsubsection{Proof of (b)} 
Fix a filling of $F$ that avoids the pattern $d\cdots 1 (d+1)$. We will show that there exists a unique $d$-RSK growth diagram on $F$ with this filling.

\emph{Uniqueness.} Suppose for the sake of contradiction that there are multiple $d$-RSK growth diagrams on $F$ with the given filling. Then there must be some lattice point $p=(x,y)$ of $F$ whose partition assignment is not uniquely determined. Choose $p$ such that $x+y$ is minimal among all lattice points with this property. Clearly $p$ does not lie on the axes. Hence, there is a cell $C$ in $F$ whose top-right corner is $p$. The partitions assigned to the other three corners of $C$ are uniquely determined (by the minimality of $x+y$). Therefore, by the forward local growth lemma there is at most one choice of partition that can be assigned to $p$ to make $C$ satisfy the $d$-RSK local rule. This is a contradiction.

\emph{Existence.} The process of constructing a growth diagram is simply a matter of iterating local growth to build up the diagram cell by cell. When constructing a $d$-RSK growth diagram from a filling there is slight complication: the forward local growth lemma has a hypothesis  (``$\kappa_d=0$ or $m=0$'') which is not obviously satisfied at each step. Verifying this hypothesis holds at each step requires us to use the fact that the filling avoids the pattern $d\cdots 1 (d+1)$. We give the details of the construction below.

First assign the empty partition to all the lattice points of $F$ on the axes. Let $C_1,\ldots, C_m$ be some enumeration of the cells of $F$ ordered in such a way that $\bigcup_{1\leq i \leq j} C_j$ is a valid Young diagram for all $1\leq j \leq m$. At step $j$ of the construction we would like to apply the forward local growth lemma at the cell $C_j$ to ``complete'' the cell. We claim that the hypotheses of the forward local growth lemma are satisfied at each step. 

Suppose we are on step $j$. By the choice of cell ordering, the left edge of $C_j$ either lies on the $y$-axis or is the right edge of a cell from a previous step. Hence, the two lattice points incident to this edge have already been assigned partitions, and those partitions interlace correctly. A similar argument applies to the bottom edge of $C_j$. Let $p$ denote the bottom-left corner of $C_j$, let $\kappa$ denote the partition assigned to $p$, and let $m$ denote the filling of $C_j$. It remains to check that $\kappa_d=0$ or $m=0$.  We consider two cases depending on whether $\Rect_p$ contains a se-chain of length $d$ or not. If it does, then $m=0$ since the filling of the diagram is assumed to avoid the pattern $d\cdots 1 (d+1)$. If it does not, then $\ell(\kappa)<d$ by Theorem \ref{thm:dRSKchains} (applied to $\Rect_p$), so $\kappa_d=0$. Hence, we have verified the hypotheses at each step. At the end of this process we have produced a growth diagram on $F$ with the desired filling.

\subsubsection{Proof of (c)}
Fix a choice of $d$-semistandard $w$-oscillating tableau assigned to the outer boundary of $F$. We will show that there exists a unique $d$-RSK growth diagram on $F$ whose outer boundary has this assignment of partitions.

\emph{Uniqueness.} Suppose for the sake of contradiction that there are multiple $d$-RSK growth diagrams on $F$ that have the required partition assignment on the outer boundary. Then there must exist a cell $C$ with bottom-left corner $p=(x,y)$ such that either the entry in $C$ is not uniquely determined or the partition assigned to $p$ is not uniquely determined. There may be multiple such $C$, so choose one such that $x+y$ is maximal. The partitions at the other three corners of $C$ are uniquely determined by the maximality of $x+y$. The backward local growth lemma implies that there is only one way to assign an entry to $C$ and a partition to $p$ to make $C$ satisfy the $d$-RSK local rule. This is a contradiction.

\emph{Existence.} As in the proof of (b), we build up the growth diagram iteratively. This time we list the cells of $F$ in the opposite order and apply the \emph{backwards} local growth lemma at each step. This is easier than the forward growth case since the only hypothesis of the backwards local growth lemma is that the partitions assigned to the three corners of the cell interlace correctly.

At the end of the process, we will have assigned partitions to all the lattice points of $F$ and entries to all the cells. It remains to check that the partitions assigned to the axes are the empty partition. For the partitions at the rightmost point on the $x$-axis and the topmost point on the $y$-axis this is true since these correspond to the first and last partitions in the oscillating tableaux assigned to the outer boundary. It then follows for the other partitions by inductively applying the interlacing property. This completes the proof.

\section{Skew $d$-RSK correspondence} \label{sec:skewdRSK}
In this section we construct another correspondence that we call the \emph{oscillating skew $d$-RSK correspondence}. This is closely related to the oscillating $d$-RSK correspondence although neither is a generalization of the other. At the end of the section we will explain how this correspondence relates to prior work of Neyman \cite{neyman2015}; Elizalde \cite{elizalde2025}; and Courtiel, Elvey Price, and Marcovici \cite{courtiel2021epm}.

\subsection{Definitions}
To define the skew $d$-RSK correspondence we must first introduce some additional definitions. We define a \emph{$d$-staircase} to be a finite sequence of integers $\lambda=(\lambda_1,\ldots,\lambda_d)$ such that $\lambda_1 \geq \cdots \geq \lambda_d$. If the $\lambda_i$'s are all nonnegative, then $\lambda$ can be identified with a $d$-partition.

All of the definitions given earlier for $d$-partitions extend naturally to $d$-staircases. For example, if $\alpha$ and $\beta$ are $d$-staircases, we write $\alpha \subseteq \beta$ if $\alpha_i \leq \beta_i$ for all $1 \leq i \leq d$. We write $\alpha \prec \beta$ if $\beta_1 \geq \alpha_1 \geq \beta_2 \geq \alpha_2 \geq \cdots \geq \beta_d \geq \alpha_d$. We write $\alpha \prec_{(d,L)} \beta$ if $\alpha \prec \beta$ and $\alpha_d \geq \beta_1 - L$. The \emph{size} of $\alpha$ is $|\alpha| = \alpha_1 + \cdots + \alpha_d$. These definitions are consistent with our earlier definitions when the $d$-staircases involved have nonnegative parts. 

A \emph{skew $d$-semistandard Young tableau of shape $\lambda/\mu$} is a sequence of $d$-staircases $\mu=\lambda^{(0)} \prec \lambda^{(1)} \prec \cdots \prec \lambda^{(r)}=\lambda$. A \emph{skew $d$-semistandard $w$-oscillating tableau of shape $\lambda/\mu$} is defined in the natural way. That is, the $\lambda^{(i)}$'s are allowed to interlace in either direction according to a type sequence $w$. For both of these definitions we call $\mu$ the \emph{inner shape} and $\lambda$ the \emph{outer shape}. The $(d,L)$-cylindric versions of these tableaux are also defined as expected (i.e. replace interlacing with $(d,L)$-interlacing), as are the definitions of \emph{minimum cylindric width} and \emph{weight}. The \emph{standard} versions of these tableaux are also defined in the obvious way.

It should be noted that a skew $d$-semistandard Young tableau of shape $\lambda/\emptyset$ is the same as a $d$-semistandard Young tableau of shape $\lambda$. However, a skew $d$-semistandard $w$-oscillating tableau of shape $\emptyset/\emptyset$ is \emph{not} the same as a $d$-semistandard $w$-oscillating tableau as the latter is not allowed to contain any partitions with negative parts.

\subsection{Skew $d$-RSK growth diagrams}
We now define the notion of a \emph{skew $d$-RSK growth diagram}. The data for such a diagram is the same as before except that each lattice point of the diagram is now assigned a $d$-staircase rather than a partition. The other difference is that we do not impose any boundary condition on these diagrams. That is, the partitions assigned to the axes are not required to be $\emptyset$. Each cell of a skew $d$-RSK growth diagram must satisfy the \emph{skew $d$-RSK local rule}, which we now state.

\medskip

\noindent\textbf{Skew $d$-RSK local rule:} Let $\kappa,\mu,\nu,\rho$ denote the $d$-staircases assigned to the corners of the cell (as labelled in Figure \ref{fig:single_cell}), and let $m$ denote the entry. It is required that $\mu \succ \kappa \prec \nu$ and $\mu \prec \rho \succ \nu$ and that $m=0$. Moreover,
\begin{equation} \label{eq:skewdRSK}
    \begin{pmatrix}
        \rho_1 \\ \rho_2 \\ \rho_3 \\ \vdots \\ \rho_d
    \end{pmatrix}
    +
    \begin{pmatrix}
        \kappa_d \\ \kappa_1 \\ \kappa_2 \\ \vdots \\ \kappa_{d-1}
    \end{pmatrix}=
    \begin{pmatrix}
        \min(\mu_d,\nu_d) + \max(\mu_1,\nu_1)\\ \min(\mu_1,\nu_1) + \max(\mu_2,\nu_2) \\ \min(\mu_2,\nu_2) + \max(\mu_3,\nu_3) \\ \vdots \\ \min(\mu_{d-1},\nu_{d-1}) + \max(\mu_d,\nu_d)
    \end{pmatrix}.
\end{equation}
\medskip

Comparing this local rule to the non-skew case, we see that the only difference (apart from the generalization to $d$-staircases) is that the entry $m$ is now required to be 0 in all cells. This means that the filling of a skew $d$-RSK growth diagram is trivial. Consequently, the correspondence we will obtain from these diagrams does not involve fillings.

\subsection{Oscillating skew $d$-RSK correspondence}

We now state our main results on skew $d$-RSK growth diagrams. The proofs of these are similar to (but easier than) the proofs we gave for the non-skew case, so we will be brief.

\begin{theorem} \label{thm:skewdRSKgrowth}
Let $w$ be a type sequence that has $r$ occurrences of `$+$' and $c$ occurrences of `$-$'. Let $\mathcal{P}^w$ denote the lattice path that starts at $(r,0)$ and ends at $(0,c)$ whose steps are encoded by $w$.
\begin{enumerate}[(a)]
    \item For any skew $d$-RSK growth diagram on $\Rect_{(r,c)}$, the sequence of $d$-staircases assigned to the points along $\mathcal{P}^w$ forms a skew $d$-semistandard $w$-oscillating tableau.
    \item Any skew $d$-semistandard $w$-oscillating tableau that is assigned to the points along $\mathcal{P}^w$ can be uniquely extended to a skew $d$-RSK growth diagram on $\Rect_{(r,c)}$.
\end{enumerate}
Consequently, there is an explicit bijection between skew $d$-RSK growth diagrams on $\Rect_{(r,c)}$ and skew $d$-semistandard $w$-oscillating tableaux.
\end{theorem}
\begin{proof}
The statement in (a) is immediate from the interlacing hypotheses in the local rule. To prove (b) one can first derive analogues of the forward and backward local growth lemmas (which are simpler than the non-skew case since the filling is trivial). These can be applied repeatedly to construct the entire diagram. This construction and the uniqueness proof are easily adapted from the proof of Theorem~\ref{thm:dRSKgrowth}.
\end{proof}

A consequence of Theorem \ref{thm:skewdRSKgrowth} is that if $w$ and $v$ are two type sequences with the same number of occurrences of `$+$' and `$-$', then there is an explicit bijection between skew $d$-semistandard $w$-oscillating tableaux and skew $d$-semistandard $v$-oscillating tableaux. This follows from the fact that both of these are in bijection with growth diagrams on the same underlying rectangle. We will collectively refer to all the bijections that arise in this way as the \emph{oscillating skew $d$-RSK correspondence}. This correspondence has the following properties.

\begin{theorem} \label{thm:skewdRSKproperties}
    Let $w$ and $v$ be type sequences. Let $T$ be a skew $d$-semistandard $w$-oscillating tableau and $P$ be a skew $d$-semistandard $v$-oscillating tableau such that $T$ and $P$ correspond to one another under the oscillating skew $d$-RSK correspondence. The following properties hold.
    \begin{enumerate}[(i)]
        \item $T$ and $P$ have the same shape. That is, they have the same inner shape and the same outer shape.
        \item $T$ and $P$ have the same weight. That is, $\wt^+(T) = \wt^+(P)$ and $\wt^-(T) = \wt^-(P)$.
        \item $T$ and $P$ have the same minimum cylindric width. That is, $\MCW_d(T) = \MCW_d(P)$.
        \item If $T^{\mathrm{rev}}$ is the reversal of $T$ (i.e. as a sequence of $d$-staircases), and $P^{\mathrm{rev}}$ is the reversal of $P$, then $T^{\mathrm{rev}}$ and $P^{\mathrm{rev}}$ correspond to each other under the oscillating skew $d$-RSK correspondence.
    \end{enumerate}
\end{theorem}
\begin{proof}
    Let $r$ and $c$ denote the number of occurrences of `$+$' and `$-$' in $w$ (and hence in $v$). By the definition of the oscillating skew $d$-RSK correspondence, the tableaux $T$ and $P$ both come from the sequences of $d$-staircases in the same growth diagram on $\Rect_{(r,c)}$. More specifically, $T$ is obtained from a certain lattice path $\mathcal{P}^w$, and $P$ is obtained from a certain lattice path $\mathcal{P}^v$. These paths start at the bottom-right corner of $\Rect_{(r,c)}$ and end at the top-left.

    Property (i) follows from the fact that $\mathcal{P}^w$ and $\mathcal{P}^v$ have the same start point and the same end point. Properties (ii) and (iii) can be proved by first considering the special case where the lattice paths $\mathcal{P}^w$ and $\mathcal{P}^v$ differ at only a single point. That is, they go in different directions around a single cell of the growth diagram. Let $\kappa,\mu,\nu,\rho$ be the $d$-staircases assigned to the corners of this cell (as labelled in Figure \ref{fig:single_cell}). Property (ii) is a direct consequence of the fact that $|\kappa| + |\rho| = |\mu| + |\nu|$, and property (iii) is a direct consequence of the fact that
    \[
        \max(\MCW_d(\kappa, \mu), \MCW_d(\kappa,\nu)) = \max(\MCW_d(\mu,\rho), \MCW_d(\nu,\rho)).
    \]
    The general case of properties (ii) and (iii) can be handled by repeatedly applying this special case. Property (iv) follows from the fact that reflecting the entire growth diagram across $y=x$ produces another growth diagram. The reflected diagram witnesses the correspondence between $T^{\mathrm{rev}}$ and $P^{\mathrm{rev}}$.
\end{proof}

Naturally, properties (ii) and (iv) of this theorem are analogous to the properties of RSK and $d$-RSK that were stated in Proposition \ref{prop:weightreflection}. That is (ii) says the oscillating skew $d$-RSK correspondence is weight-preserving, and (iv) says it commutes with reflection. 

A nice enumerative consequence of Theorem \ref{thm:skewdRSKproperties} is that the number of skew $d$-semistandard $w$-oscillating tableaux of a given shape, weight, and minimum cylindric width is independent of the choice of type sequence $w$ (assuming we fix the number of occurrences of `$+$' and `$-$').

\subsection{Special cases and connections to prior work} \label{subsec:skewdRSKconnections}

The oscillating skew $d$-RSK correspondence can be restricted in various ways to obtain other correspondences. In particular, by placing an upper bound $L$ on the minimum cylindric width of the tableaux, we obtain a correspondence involving $(d,L)$-cylindric tableaux. Moreover, if we take the type sequences $w$ and $v$ to be of the form $w=+^r -^c$ and $v=-^c +^r$, we obtain a non-oscillating version of this cylindric correspondence between pairs of skew $(d,L)$-cylindric semistandard Young tableaux. Such a correspondence was constructed previously by Neyman \cite{neyman2015} using a variant of Schensted insertion. Our work gives an oscillating generalization of Neyman's correspondence.

In the case of cylindric \emph{standard} tableaux, Neyman's bijection was reformulated by Elizalde in \cite{elizalde2025} using growth diagrams. Our local rule agrees with Elizalde's in this setting (see Section \ref{sec:connections}). Elizalde also used the growth diagram construction to describe an oscillating cylindric correspondence. He showed that this correspondence is equivalent to a certain correspondence constructed by Courtiel, Elvey Price, and Marcovici \cite{courtiel2021epm} involving walks in a $d$-dimensional simplex.

We may summarize the contributions of this section as follows. Firstly, we showed that a simple modification of the $d$-RSK local rule leads to \emph{skew} correspondences in a natural way. Secondly, the cylindric correspondences that were considered by Neyman and Elizalde are actually restrictions of broader correspondences that are independent of the $L$ parameter. Thirdly, the cylindric correspondences that were considered by Neyman and Elizalde can be naturally extended to the semistandard oscillating setting.

\begin{remark}
After the first version of this paper was posted to the arXiv, Elizalde informed us that he had independently extended his work in \cite{elizalde2025} to the semistandard setting. This will appear in \cite{elizalde2026}. The local rule he uses for this is identical to our skew $d$-RSK local rule apart from the difference in perspective mentioned above (i.e. our correspondences are independent of the $L$ parameter).
\end{remark}

\begin{remark}
The skew $d$-RSK correspondence can also be used to give bijective proofs of certain identities involving skew Schur functions and cylindric Schur functions (see \cite[Thm. 5.14, Cor. 5.21]{neyman2015} for the cylindric case). In the non-cylindric setting, the resulting skew Schur function identity is equivalent to a certain relation involving sums of Littlewood--Richardson coefficients. This relation was first proved by Coquereaux and Zuber \cite[Thm.~1]{coquereaux2011z} using representation theory.
\end{remark}

\section{Conjugation, cointerlacing, and row-strict tableaux} \label{sec:dual}
In this section we discuss some dual notions in the theory of cylindric tableaux. We start by reviewing relevant facts and definitions for ordinary partitions and tableaux, and then we describe their cylindric analogues. At the end of the section we give an example of how these ideas can be combined with cylindric correspondences to obtain further bijections.

\subsection{Ordinary partitions and tableaux}
Given two partitions $\lambda$ and $\mu$, we say that \emph{$\lambda$ cointerlaces $\mu$}, denoted $\lambda \prec' \mu$, if $\mu_i-\lambda_i \in \{0,1\}$ for all $i \geq 1$. Like interlacing, cointerlacing is a non-transitive relation on partitions that is stronger than the $\subseteq$ relation. There is an important special case in which all three of these relations coincide: if $|\mu| - |\lambda| \in \{0,1\}$, then $\lambda \prec \mu \iff \lambda \subseteq \mu  \iff \lambda \prec' \mu$.

A \emph{row-strict Young tableau} is a sequence of partitions $\emptyset=\lambda^{(0)} \prec' \lambda^{(1)}\prec' \ldots \prec' \lambda^{(k)}$. This is the same as a semistandard Young tableau except that interlacing has been replaced with cointerlacing. We define \emph{weight} and \emph{shape} in the same way as the semistandard case.

Obviously, the notions of semistandardness and row-strictness are different, but they coincide when the weight vector of a tableau is in $\{0,1\}^k$. This follows from our earlier observation about interlacing and cointerlacing. In particular, the notion of a \emph{standard} Young tableau could be equivalently defined as either a semistandard or row-strict tableau of weight $(1,1,\ldots,1)$.

Given a partition $\lambda$, the \emph{conjugate of $\lambda$}, denoted $\lambda'$, is the partition whose Young diagram is obtained by reflecting the Young diagram of $\lambda$ across the $y=x$ line. This operation has the following well-known properties.

\begin{proposition} \label{prop:conjugation}
    Let $\lambda$ and $\mu$ be partitions. Then the following properties hold:
    \begin{enumerate}[(i)]
        \item $(\lambda')' = \lambda$,
        \item $|\lambda'| = |\lambda|$,
        \item $\lambda \subseteq \mu \iff \lambda' \subseteq \mu'$,
        \item $\lambda \prec' \mu \iff \lambda' \prec \mu'$.
    \end{enumerate}
\end{proposition}

Given any sequence of partitions $T=(\lambda^{(0)},\ldots,\lambda^{(k)})$ we define the \emph{conjugate of $T$}, denoted $T'$, to be the sequence of partitions $((\lambda^{(0)})',\ldots,(\lambda^{(k)})')$. The following result follows immediately from Proposition \ref{prop:conjugation}.

\begin{proposition}
    The conjugation operation $T\mapsto T'$ is an involution on sequences of partitions. It restricts to a weight-preserving bijection between semistandard Young tableaux and row-strict Young tableaux, and it restricts further to an involution on standard Young tableaux.
\end{proposition}

\begin{remark}
    Another natural kind of tableau can be defined using the $\subseteq$ relation instead of interlacing or cointerlacing. These are the \emph{reverse plane partitions}. We will not have any particular use for these objects, but conjugation also gives a weight-preserving involution on them. Standard Young tableaux are reverse plane partitions of weight $(1,1,\ldots,1)$.
\end{remark}

\subsection{Cylindric analogues} \label{sec:dualcyl}
We can extend the definition of cointerlacing to $d$-staircases in the obvious way: if $\lambda$ and $\mu$ are $d$-staircases, we say that $\lambda \prec' \mu$ if $\mu_i - \lambda_i \in \{0,1\}$ for all $1 \leq i \leq d$. A \emph{skew $d$-row-strict Young tableau of shape $\lambda/\mu$} is a sequence of $d$-staircases $\mu=\lambda^{(0)} \prec' \lambda^{(1)} \prec' \cdots \prec' \lambda^{(k)}=\lambda$. As in the non-skew case, the notions of semistandardness and row-strictness coincide when the weight vector is in $\{0,1\}^k$.

We would like to extend conjugation to $d$-staircases as well. However, to get a reasonable analogue, we must restrict to a special class of $d$-staircases. For any pair of positive integers $(d,L)$, we define a \emph{$(d,L)$-staircase} to be a $d$-staircase $\lambda$ such that $\lambda_1 - \lambda_d \leq L$. If $\lambda_i \geq 0$ for all $1\leq i \leq d$, we call it a \emph{$(d,L)$-partition}. For two $d$-staircases $\lambda$ and $\mu$, we say that $\lambda$ \emph{$(d,L)$-cointerlaces} $\mu$, denoted $\lambda \prec'_{(d,L)} \mu$, if $\lambda$ and $\mu$ are both $(d,L)$-staircases and if $\lambda \prec' \mu$. 

\begin{remark} \label{rem:interlacingcointerlacing}
    There is a subtle point that is worth emphasizing regarding a certain distinction between interlacing and cointerlacing. For cointerlacing, if $\lambda$ and $\mu$ are $d$-staircases
    \[ \lambda \prec'_{(d,L)} \mu \iff \Big(\lambda \text{ and } \mu \text{ are } (d,L)\text{-staircases, and } \lambda \prec' \mu\Big). \]
    This is true by definition. For interlacing, we have
    \[ \lambda \prec_{(d,L)} \mu \implies \Big(\lambda \text{ and } \mu \text{ are } (d,L)\text{-staircases, and } \lambda \prec \mu\Big), \]
    but the converse does \emph{not} hold. Indeed, $(d,L)$-interlacing requires $\mu_1 - \lambda_d \leq L$ which does not follow from the fact that $\lambda$ and $\mu$ are $(d,L)$-staircases and $\lambda \prec \mu$.
\end{remark}

The usefulness of the notion of $(d,L)$-staircases comes from the fact that these objects have a diagrammatic representation that is analogous to Young diagrams for ordinary partitions. This is described in the following proposition.

\begin{proposition}
    Let $(d,L)$ be a pair of positive integers. There are explicit bijective correspondences between the following collections of objects:
    \begin{enumerate}
        \item $(d,L)$-staircases,
        \item bi-infinite weakly decreasing sequences of integers $(a_i)_{i \in \Z}$ such that $a_{i+d} = a_i - L$ for all $i \in \Z$,
        \item bi-infinite lattice paths in $\Z^2$ that consist of unit steps up or left and that are invariant under translating the path by $(-L,d)$.
    \end{enumerate}
\end{proposition}
\begin{proof}
If $\lambda$ is a $(d,L)$-staircase, the bi-infinite sequence is obtained by setting $a_i = \lambda_i$ for $1 \leq i \leq d$ and then extending to all $i\in\Z$ using the relation $a_{i+d} = a_i - L$. 

Given a bi-infinite sequence $(a_i)_{i \in \Z}$, the lattice path is obtained as follows. Start at $(0, a_1)$, then take one step up and $a_1-a_2$ steps left, then one step up and $a_2-a_3$ steps left, and so on. This determines the portion of the lattice path lying above the $x$-axis. The rest of the path is determined by translation invariance. See Figure \ref{fig:cylshape} for an illustration. It is straightforward to verify that these maps are bijections. 
\end{proof}

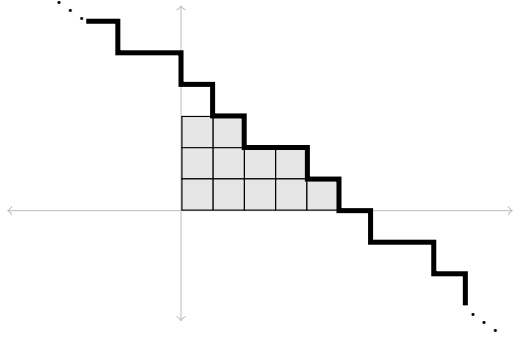
\begin{figure}
        \centering
        \ytableausetup{boxsize=0.40cm}
        \begin{tikzpicture}[scale=.7]
            \pgfmathsetmacro{\boxsize}{0.6}
            \draw[<->,lightgray] (-8*\boxsize,-1.5*\boxsize) -- ++(16*\boxsize,0);
            \draw[<->,lightgray] (-2.5*\boxsize,-5*\boxsize) -- ++(0,10*\boxsize);
            \node at (-6*\boxsize,5.1*\boxsize) {$\ddots$};
            \node at (7.1*\boxsize,-4.8*\boxsize) {$\ddots$};
            \node[inner sep=0pt] (yd) {\ydiagram[*(gray!20)]{2,4,5}};
            \draw[line width=1.9 pt] (6.5*\boxsize,-4.5*\boxsize) 
                -- ++(0,\boxsize) -- ++(-\boxsize,0) -- ++(0,\boxsize)
                -- ++(-\boxsize,0) -- ++(-\boxsize,0) -- ++(0,\boxsize) 
                -- ++(-\boxsize,0)
                -- ++(0,\boxsize) -- ++(-\boxsize,0) -- ++(0,\boxsize)
                -- ++(-\boxsize,0) -- ++(-\boxsize,0) -- ++(0,\boxsize) 
                -- ++(-\boxsize,0)
                -- ++(0,\boxsize) -- ++(-\boxsize,0) -- ++(0,\boxsize)
                -- ++(-\boxsize,0) -- ++(-\boxsize,0) -- ++(0,\boxsize) 
                -- ++(-\boxsize,0);
        \end{tikzpicture}
        \caption{The Young diagram of $\lambda=(5,4,2)$ and the corresponding lattice path when viewed as a $(3,4)$-staircase.}
        \label{fig:cylshape}
\end{figure}

The lattice path interpretation of $(d,L)$-staircases lets us define \emph{cylindric conjugation} for these objects in a natural way. We simply reflect the path across the line $y=x$. This induces an operation $\lambda \mapsto \tr_{(d,L)}(\lambda)$ that maps $(d,L)$-staircases to a $(L,d)$-staircases. We also obtain a direct analogue of Proposition~\ref{prop:conjugation} for $(d,L)$-staircases.

\begin{proposition} \label{prop:cyl_conjugation}
    For all $(d,L)$-staircases $\lambda$ and $\mu$, the following properties hold:
    \begin{enumerate}[(i)]
        \item $\tr_{(L,d)}(\tr_{(d,L)}(\lambda)) = \lambda$,
        \item $|\tr_{(d,L)}(\lambda)| = |\lambda|$,
        \item $\lambda \subseteq \mu \iff \tr_{(d,L)}(\lambda) \subseteq \tr_{(d,L)}(\mu)$,
        \item $\lambda \prec'_{(d,L)} \mu \iff \tr_{(d,L)}(\lambda) \prec_{(L,d)} \tr_{(d,L)}(\mu)$.
    \end{enumerate}
\end{proposition}
\begin{proof}
    Property (i) is true by definition. To prove property (ii), first observe that it is true when $\lambda_d = 0$ since $\tr_{(d,L)}(\lambda)=\lambda'$ in this case. The general case can be reduced to the $\lambda_d=0$ case by deriving relations for how $|\lambda|$ and $\tr_{(d,L)}(\lambda)$ are affected by translations of the lattice path. The same strategy works to prove (iii). The proof of (iv) is quite similar to the non-cylindric case: there is a natural geometric characterization of $(L,d)$-interlacing as shown in Figure~\ref{fig:cylinterlace}. The result follows.
\end{proof}

\begin{figure}
    \ytableausetup{boxsize=normal}
    \begin{subfigure}{.45\textwidth}
        \centering
        \begin{tikzpicture}
            \node {\ydiagram{0,0,0,3+2,3+4,3+5}*[*(gray!40)]{2+2,4+1,5+1,5+2,7+1,8+1}};
            \pgfmathsetmacro{\boxsize}{0.6}
            \draw[line width=1.9 pt,blue] (3.5*\boxsize,-3*\boxsize) 
                -- ++(0,\boxsize) -- ++(-\boxsize,0) -- ++(0,\boxsize)
                -- ++(-\boxsize,0) -- ++(-\boxsize,0) -- ++(0,\boxsize) 
                -- ++(0,\boxsize) -- ++(-\boxsize,0) -- ++(0,\boxsize)
                -- ++(-\boxsize,0) -- ++(-\boxsize,0) -- ++(0,\boxsize);
            \draw[line width=1.9 pt,red] (4.5*\boxsize,-3*\boxsize) 
                -- ++(0,\boxsize) -- ++(-\boxsize,0) -- ++(0,\boxsize)
                -- ++(-\boxsize,0) -- ++(0,\boxsize) -- ++(-\boxsize,0)
                -- ++(0,\boxsize) -- ++(-\boxsize,0) -- ++(0,\boxsize)
                -- ++(-\boxsize,0) -- ++(0,\boxsize) -- ++(-\boxsize,0);
        \end{tikzpicture}
        \caption{}
    \end{subfigure}
    \begin{subfigure}{.45\textwidth}
        \centering
        \begin{tikzpicture}
            \node {\ydiagram{0,0,0,4+2,4+4,4+5}*[*(gray!40)]{2+2,4+1,5+1,6+2,8+1,9+1}};
            \pgfmathsetmacro{\boxsize}{0.6}
            \draw[line width=1.9 pt,blue] (4*\boxsize,-3*\boxsize) 
                -- ++(0,\boxsize) -- ++(-\boxsize,0) -- ++(0,\boxsize)
                -- ++(-\boxsize,0) -- ++(-\boxsize,0) -- ++(0,\boxsize)
                -- ++(-\boxsize,0) 
                -- ++(0,\boxsize) -- ++(-\boxsize,0) -- ++(0,\boxsize)
                -- ++(-\boxsize,0) -- ++(-\boxsize,0) -- ++(0,\boxsize)
                -- ++(-\boxsize,0);
            \draw[line width=1.9 pt,red] (5*\boxsize,-3*\boxsize) 
                -- ++(0,\boxsize) -- ++(-\boxsize,0) -- ++(0,\boxsize)
                -- ++(-\boxsize,0) -- ++(0,\boxsize) -- ++(-\boxsize,0)
                -- ++(-\boxsize,0)
                -- ++(0,\boxsize) -- ++(-\boxsize,0) -- ++(0,\boxsize)
                -- ++(-\boxsize,0) -- ++(0,\boxsize) -- ++(-\boxsize,0)
                -- ++(-\boxsize,0);
        \end{tikzpicture}
        \caption{}
    \end{subfigure}
    \caption{\textsc{(a)} The lattice paths corresponding to the $(3,3)$-partitions $\lambda=(5,4,2)$ and $\mu=(6,5,4)$ are shown in blue and red. These partitions do not $(3,3)$-interlace because the region between them contains multiple cells in the same column. \\
    \textsc{(b)} The lattice paths of $\lambda$ and $\mu$ are shown as $(3,4)$-partitions. The region between them has at most one cell in each column, so $\lambda \prec_{(3,4)} \mu$.}
    \label{fig:cylinterlace}
\end{figure}
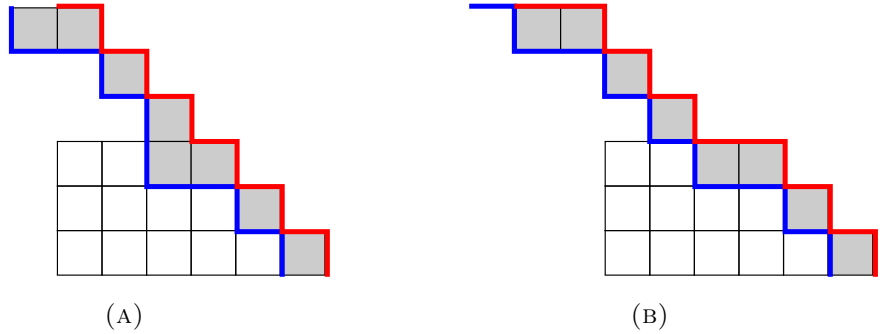

Properties (i)--(iii) of this proposition be found in work of Goodman and Wenzl \cite[Lem. 1.5]{goodman1990w}, although their definitions are somewhat different than ours. Their version is also only stated for $(d,L)$-partitions. Property (iv) of this proposition is important because it gives the desired duality between interlacing and cointerlacing for $(d,L)$-staircases.

For sequences of $(d,L)$-staircases, we define a cylindric conjugation map $T \mapsto \tr_{(d,L)}(T)$ just like in the non-cylindric case (i.e. apply $\tr_{(d,L)}$ to each term in the sequence). A \emph{skew $(d,L)$-cylindric row-strict Young tableau of shape $\lambda/\mu$} is a sequence of $d$-staircases $\mu=\lambda^{(0)} \prec'_{(d,L)} \lambda^{(1)} \prec'_{(d,L)} \cdots \prec'_{(d,L)} \lambda^{(k)} = \lambda$. When $\mu=\emptyset$, we remove the word ``skew'' and simply say that these tableaux have shape $\lambda$. These definitions are exact analogues of the semistandard case in the sense that $(d,L)$-interlacing has been replaced with $(d,L)$-cointerlacing. We define the \emph{weight} of such these tableaux in the usual way. The following proposition gives the desired duality between row-strict tableaux and semistandard tableaux in the cylindric setting.

\begin{proposition}[{cf. \cite[Prop. 2.4]{huh2025kkc}}] \label{prop:cyl_conjugation_tableaux}
    The conjugation operation $T\mapsto \tr_{(d,L)}(T)$ is a bijection between sequences of $(d,L)$-staircases and sequences of $(L,d)$-staircases. It restricts to a weight-preserving bijection between skew $(d,L)$-cylindric semistandard Young tableaux and skew $(L,d)$-cylindric row-strict Young tableaux. It further restricts to a bijection between skew $(d,L)$-cylindric standard Young tableaux and skew $(L,d)$-cylindric standard Young tableaux. The same statements also hold in the non-skew setting.
\end{proposition}
\begin{proof}
    The statements here are consequences of Proposition \ref{prop:cyl_conjugation}. Note that if $T$ is a $(d,L)$-cylindric semistandard Young tableau, then the constituent $d$-staircases of $T$ are all $(d,L)$-staircases (see Remark \ref{rem:interlacingcointerlacing}). The same is true for row-strict tableaux. Hence, the conjugation operation is well-defined on these objects and has the desired properties.
\end{proof}

\begin{remark}
    By Remark \ref{rem:interlacingcointerlacing}, there is an appealingly simple analogue of ``minimum cylindric width'' in the row-strict setting. Indeed, given a $d$-row-strict tableau $T$, we can define its minimum cylindric width to be the minimum $L$ such that all of the constituent partitions of $T$ are $(d,L)$-partitions. Unfortunately, if $T$ is both $d$-semistandard and $d$-row-strict then the two notions of minimum cylindric width do not necessarily coincide (consider $T=(\emptyset,(1,0), (2,1))$ as a sequence of $2$-staircases). They do coincide if $T$ is standard, however.
\end{remark}

\subsection{Duality and RSK-type correspondences}
By combining conjugation with the various correspondences discussed in previous sections, it is possible to obtain further bijections involving row-strict tableaux. We will give a representative example of this type of result using the oscillating skew $d$-RSK correspondence.

To state this result we must define oscillating tableaux in the row-strict setting. It suffices to say that we can take all of the definitions from the semistandard oscillating setting and carry them over to the row-strict oscillating setting by replacing interlacing with cointerlacing and $(d,L)$-interlacing with $(d,L)$-cointerlacing. With these definitions in place, we can state the following result.

\begin{theorem} \label{thm:skewdRSKrowstrict}
    Let $(d,L)$ be a pair of positive integers. Let $w$ and $v$ be type sequences that have the same number of occurrences of `$+$' and the same number of occurrences of `$-$'. There is a shape and weight preserving bijection between skew $(d,L)$-cylindric row-strict $w$-oscillating tableaux and skew $(d,L)$-cylindric row-strict $v$-oscillating tableaux. This correspondence commutes with reversal of the tableaux.
\end{theorem}
\begin{proof}
    By Theorem \ref{thm:skewdRSKproperties}, the oscillating skew $L$-RSK correspondence restricts to a shape and weight preserving bijection between skew $(L,d)$-cylindric semistandard $w$-oscillating tableaux and skew $(L,d)$-cylindric semistandard $v$-oscillating tableaux. Composing this with cylindric conjugation gives the desired bijection involving row-strict tableaux.
\end{proof}

It should also be noted that taking $L\to\infty$ in Theorem \ref{thm:skewdRSKrowstrict} leads to a nontrivial enumerative consequence involving skew $d$-row-strict $w$-oscillating tableaux. That is, the number of such tableaux of a given shape and weight is independent of the choice of type sequence $w$ (assuming we fix the number of occurrences of `$+$' and `$-$'). This latter statement does not involve cylindric objects, but we have obtained it as a consequence of cylindric correspondences. 

\section{Bijections between fillings of Young diagrams} \label{sec:fillings_cors}
We now discuss some additional bijections which do not involve tableaux. The following result is obtained by combining the oscillating RSK and $d$-RSK correspondences.

\begin{theorem} \label{thm:fillings_genshape}
    Let $F$ be a Young diagram. There is an explicit bijective correspondence between 
    \begin{enumerate}
        \item fillings of $F$ that avoid the pattern $d\cdots 1 (d+1)$, and
        \item fillings of $F$ such that for every lattice point $p$ of $F$ the rectangular region $\Rect_p$ does not contain a se-chain of length $d+1$.
    \end{enumerate}
    This correspondence preserves all row sums and column sums, and it commutes with reflection of these fillings across the $y=x$ line.
\end{theorem}
\begin{proof}
    Let $w$ denote the type sequence that encodes the outer boundary of $F$. Both sets of fillings are in bijection with $d$-semistandard $w$-oscillating tableaux by the oscillating $d$-RSK and RSK correspondences respectively. The auxiliary properties follow from Proposition~\ref{prop:weightreflection}.
\end{proof}

Both sets of fillings in this theorem are infinite, so the statement is only enumeratively meaningful when a bound is imposed on the row and column sums. In the special case where the row and column sums are all equal to 1, the existence of a bijection between such fillings was proved by Backelin, West, and Xin \cite[Prop. 3.1]{backelin2007}. Our bijection coincides with theirs in this special case, although this is not at all obvious from the constructions. This follows from work of Bloom and Saracino \cite{bloom2012s}. We will say more about this in Section \ref{sec:connections}.

The fact that the Backelin--West--Xin bijection commutes with reflection across the $y=x$ line was first proved by Bousquet-Mélou and Steingrímsson \cite{bousquet2005}. Bloom and Saracino gave a different proof of this fact by showing that the Backelin--West--Xin bijection can be reformulated in terms of growth diagrams and then using the natural symmetry of these diagrams. Our more general bijection commutes with reflection across the $y=x$ line for the same reason.

For completeness, we also state the special case of Theorem \ref{thm:fillings_genshape} where $F$ is a rectangular Young diagram. In this case the condition on se-chains becomes quite simple.

\begin{corollary} \label{cor:fillings_square}
    Let $F$ be a rectangular Young diagram. There is an explicit bijective correspondence between 
    \begin{enumerate}
        \item fillings of $F$ that avoid the pattern $d\cdots 1 (d+1)$, and
        \item fillings of $F$ that do not contain a se-chain of length $d+1$.
    \end{enumerate}
    This correspondence preserves all row sums and column sums, and it commutes with reflection of these fillings across the $y=x$ line.
\end{corollary}

In the next section we will give a refinement of this bijection in the special case where the fillings have row and column sums equal to 1.

\section{Pattern-avoiding permutations} \label{sec:permutations}
Throughout this section we use the following notation: for positive integers $n, d, L$, let
\begin{align*}
S_n^{(d,L)} &= \{\pi \in S_n : \pi \text{ avoids } d\cdots 1 (d+1) \text{ and } 1\cdots (L+1)\},\\
S_n^{(d,\infty)} &= \{\pi \in S_n : \pi \text{ avoids } d\cdots 1 (d+1)\},\\
S_n^{(\infty,L)} &= \{\pi \in S_n : \pi \text{ avoids } 1\cdots (L+1)\}, \\
S_n^{*(d)} &= \{\pi \in S_n : \pi \text{ avoids } (d+1) \cdots 1\}.\\
\end{align*}
Let $I_n^{(d,L)}$, $I_n^{(d,\infty)}$, $I_n^{(\infty,L)}$, and $I_n^{*(d)}$ denote the subsets of these sets consisting of involutions.

\subsection{Wilf-equivalences}
We start by stating a known result, which is a special case of the work of Backelin et al. and Bousquet-Mélou and Steingrímsson discussed in the previous section (plus an additional idea of Krattenthaler). We will provide a proof using the machinery we have set up.

\begin{corollary}[{see \cite[Prop. 3.1]{backelin2007},\cite[Thm. 4]{bousquet2005}, \cite[Thm. 14]{krattenthaler2006}}] \label{cor:wilf_infinite}
    For all $n, d\in\N$, we have 
    \begin{align*}
        |S_n^{(d,\infty)}| &= |S_n^{*(d)}| = |S_n^{(\infty,d)}|, \text{ and } \\
        |I_n^{(d,\infty)}| &= |I_n^{*(d)}| = |I_n^{(\infty,d)}|.
    \end{align*}
\end{corollary}
\begin{proof}
    The first equality on each line is a special case of Corollary \ref{cor:fillings_square} when the row and column sums of fillings are equal to 1. These equalities are due to Backelin et al. and Bousquet-Mélou and Steingrímsson, respectively. 

    The second equality on the first line is trivial: given a permutation without an increasing subsequence of length $d+1$, reversing it gives a permutation without a decreasing subsequence of length $d+1$ and vice versa. This proof breaks down for involutions, however. There is a different approach which does generalize. Given a permutation $\pi\in S_n^{*(d)}$, apply the RS correspondence to get a pair of standard Young tableaux $(P,Q)$. Since $\pi$ avoids $(d+1)\cdots 1$, the shape of $P$ and $Q$ has at most $d$ rows. Conjugating these tableaux gives a pair whose shape has at most $d$ columns. This pair then corresponds to a permutation in $S_n^{(\infty,d)}$ by the RS correspondence. All of these steps are reversible so we obtain a bijection between $S_n^{*(d)}$ and $S_n^{(\infty,d)}$. For involutions, the same argument goes through. The pairs of tableaux are of the form $(P,P)$ in this particular case. This idea of combining RS with conjugation was used by Krattenthaler \cite{krattenthaler2006} to prove several results along these lines.
\end{proof}

In the language of pattern-avoiding permutations (see \cite{kitaev2011}), the first line of equalities in Corollary~\ref{cor:wilf_infinite} says that the patterns $d\cdots 1 (d+1)$, $(d+1)\cdots 1$, and $1\cdots (d+1)$ are all \emph{Wilf-equivalent}. The bijections that we used in the proof can also clearly be generalized to handle fillings of non-square Young diagrams with row and column sums equal to 1. The notion of Wilf-equivalence in this more general setting is called \emph{shape-Wilf-equivalence} (see \cite{backelin2007} for the precise definition).

We now state another infinite family of Wilf-equivalences. This can be viewed as an $L<\infty$ refinement of the previous corollary.

\begin{corollary} \label{cor:wilf}
    For all $n, d, L\in \N$, we have $|S_n^{(d,L)}| = |S_n^{(L,d)}|$ and $|I_n^{(d,L)}| = |I_n^{(L,d)}|$.
\end{corollary}
\begin{proof}
    By the cylindric RS correspondence, permutations in $S_n^{(d,L)}$ are in bijection with pairs of $(d,L)$-cylindric standard Young tableaux that have a common shape and are of size $n$. Cylindric conjugation turns this into a pair of $(L,d)$-cylindric standard Young tableaux, which then correspond to permutations in $S_n^{(L,d)}$ under the cylindric RS correspondence. This bijection also restricts to involutions by the same argument as the previous proof.
\end{proof}

This corollary establishes the Wilf-equivalence of the pattern sets $\{d\cdots 1 (d+1), 1\cdots (L+1)\}$ and $\{L \cdots 1 (L+1), 1 \cdots (d+1)\}$. The proof also easily generalizes to non-square Young diagrams, which shows that these pattern sets are shape-Wilf-equivalent as well. This Wilf-equivalence does not seem to have appeared in the literature before. 

It is worth noting that there is another proof of Corollary \ref{cor:wilf} that uses previously available tools rather than the cylindric RS correspondence. We briefly summarize this alternative approach now.

\begin{proof}[Alternative proof of Corollary \ref{cor:wilf}]
    First, by combining the oscillating RS correspondence with conjugation, one can prove the shape-Wilf-equivalence of the pattern sets $\{d\cdots 1, 1\cdots L\}$ and $\{L \cdots 1, 1 \cdots d\}$. This result is essentially stated in \cite[Thm. 3]{krattenthaler2006}, but not in the language of pattern-avoidance.

    Next, it is a general fact that shape-Wilf-equivalence is well-behaved with respect to ``exchanging prefixes''. In particular, the shape-Wilf-equivalence given in the previous paragraph automatically implies the shape-Wilf-equivalence of the pattern sets $\{d\cdots 1 (d+1), 1\cdots L(L+1)\}$ and $\{L \cdots 1 (L+1), 1 \cdots d(d+1)\}$. This prefix-exchange property was first proved in \cite[Thm. 2.3]{backelin2007} for singleton sets of patterns, but the proof easily carries over to sets with multiple patterns (see \cite[Thm. 2]{burstein2025hkz}). It can also be adapted to the setting of symmetric fillings of Young diagrams (see \cite[Thm. 3.1]{jaggard2002}), which proves the involution case of the corollary.
\end{proof}

\subsection{Asymptotics} \label{sec:asymptotics}
In this section we derive asymptotics for $|S_n^{(d,L)}|$ as $n\to \infty$. For the sake of comparison, we first state the analogous result for $|S_n^{(\infty,L)}|$, which is due to Regev \cite{regev1981}.

\begin{theorem}[{see \cite[Thm. 7]{stanley2007}}] \label{thm:regev}
For any fixed positive integer $L$, we have
\[|S_n^{(\infty,L)}| \sim C_{L}\frac{L^{2n}}{n^{(L^2-1)/2}} \;\text{ as } n\to \infty,
\]
where 
\[
C_L = 1! 2! \cdots (L-1)! \left(\frac{1}{\sqrt{2\pi}}\right)^{L-1} \left(\frac{1}{2}\right)^{(L^2-1)/2} L^{L^2/2}.
\]
\end{theorem}

This asymptotic also implies the much weaker estimate $|S_n^{(\infty,L)}| = L^{2n+o(1)}$ as $n\to \infty$. Thus, the number of permutations in $S_n$ that avoid $1\cdots (L+1)$ grows exponentially as $n\to\infty$, and the growth rate is $L^2$. In the language of pattern-avoiding permutations this growth rate is called the \emph{Stanley--Wilf limit} (see \cite[Sec. 6.1.4]{kitaev2011}) of the pattern $1\cdots (L+1)$.

One approach to proving Theorem \ref{thm:regev} is to use a certain connection between $|S_n^{(\infty,L)}|$ and random matrix theory. A theorem of Rains \cite{rains1998} says that $|S_n^{(\infty,L)}|$ is equal to the expected value of $|\tr(U)|^{2n}$ where $U$ is a random $L\times L$ matrix distributed according to Haar measure on the unitary group. (This random matrix distribution is also known as the \emph{circular unitary ensemble}.) The expected value of $|\tr(U)|^{2n}$ can be written as an integral over the $L$-torus using the Weyl integration formula, and the asymptotics can then be derived using standard techniques.

The method described above is quite different from Regev's approach, but it is important from our perspective because we will use an analogous method to derive asymptotics for $|S_n^{(d,L)}|$. The following theorem is the main result of this section.

\begin{theorem} \label{thm:asymptotics}
Fix $d, L \in \N$, and let $M = d + L$. We have
\[
|S_n^{(d,L)}| \sim C_{d,L} \cdot \left(\frac{\sin(\pi d/M)}{\sin(\pi/M)}\right)^{2n} \;\text{ as } n\to \infty,
\]
where
\[C_{d,L} = \frac{1}{M^{d-1}} \prod_{j = 1}^{d-1} \left(4 \sin^2(\pi j/M)\right)^{d-j}.\]
\end{theorem}

To prove Theorem~\ref{thm:asymptotics}, we will need the following result from our work in \cite{dobner2026}.

\begin{theorem}[\cite{dobner2026}] \label{thm:patternavoidcount}
    Let $d,L,n \in \N$, and let $M = d+L$. Let $e(x) = e^{2\pi i x}$. Then the quantity
    \begin{equation}\label{eq:patternavoidcount}
    \frac{1}{M^d} \frac{1}{d!}\sum_{t_1,\ldots,t_d=1}^{M } \left|\sum_{l=1}^d e(\tfrac{t_l}{M})\right|^{2n} \prod_{1\leq j < k \leq d} |e(\tfrac{t_k}{M})-e(\tfrac{t_j}{M})|^2
    \end{equation}
    is equal to the number of pairs $(P,Q)$ of $(d,L)$-cylindric standard Young tableaux with $\sh(P) = \sh(Q)$ and $|\sh(P)| = n$.
\end{theorem}

Combining this theorem with the cylindric RS correspondence, we see that $|S_n^{(d,L)}|$ is equal to the expression in \eqref{eq:patternavoidcount}. The latter expression can be viewed as the expected value of $|\tr(U)|^{2n}$ where $U$ is a random $d\times d$ matrix drawn from a certain discrete analogue of the circular unitary ensemble. Hence, this is an exact analogue of the theorem of Rains stated above involving $|S_n^{(\infty,L)}|$.

Before giving the proof of Theorem \ref{thm:asymptotics}, let us make a few further remarks about how we discovered the connection between $|S_n^{(d,L)}|$ and random matrix theory in the first place. Prior to proving any of the results in this paper, we were investigating the sum in \eqref{eq:patternavoidcount} for other reasons. In particular, we were interested in discrete analogues of the circular unitary ensemble because of a certain connection to analytic number theory (see \cite{dobner2026} for details). While computing a table of values for these sums (for various $d$,$L$, and $n$), we discovered a match with the OEIS sequence A047849 \cite{oeis}. This sequence counts permutations in $S_n$ avoiding the patterns $3214$ and $1234$. This match was intriguing due to the similarity with the aforementioned theorem of Rains.

Upon doing further calculations, we were able to make a guess at the correct connection between our random matrix sum and pattern-avoiding permutations. In trying to prove this connection, we discovered Theorem \ref{thm:patternavoidcount} and the cylindric RS correspondence. Our results on the random matrix theory side of these investigations are detailed in \cite{dobner2026}.

\begin{proof}[Proof of Theorem \ref{thm:asymptotics}]
As noted above, Theorem \ref{thm:patternavoidcount} and the cylindric RS correspondence imply that \eqref{eq:patternavoidcount} is equal to $|S_n^{(d,L)}|$. Hence, our goal is to understand the behavior of \eqref{eq:patternavoidcount} as $n \to \infty$. If $d$ and $L$ are fixed, then \eqref{eq:patternavoidcount} can be rewritten as a sum over a fixed collection of terms of the form $B r^{2n}$, where each $B$ and $r$ depend on the values $t_1,\ldots,t_d$. In particular, for a given choice of $t_1,\ldots,t_d$ the values of $B$ and $r$ are given by
\[
r= \left| \sum_{l=1}^d e(\tfrac{t_l}{M}) \right|\; \text{ and }\; B = \frac{1}{M^d} \frac{1}{d!} \prod_{1 \leq j < k \leq d} |e(\tfrac{t_k}{M}) - e(\tfrac{t_j}{M})|^2.
\]
We must identify the configurations of $t_1,\ldots,t_d\in\{1,\ldots,M\}$ that maximize $r$ subject to the constraint that $B \neq 0$. These give the dominant terms in \eqref{eq:patternavoidcount} as $n \to \infty$.

The value of $B$ is nonzero if and only if the $t_j$'s are distinct. On the other hand, the value of $r$ is largest when the $e(\tfrac{t_l}{M})$'s are clustered together as closely as possible on the unit circle. Hence, it is intuitively clear that the (constrained) maximum value of $r$ should occur precisely when the $t_j$'s are consecutive integers modulo $M$ rearranged in any order. To give a rigorous proof of this claim, one can check that for any fixed $\theta\in\R$ it is possible to increase the component of $\sum_{l=1}^d e(\tfrac{t_l}{M})$ in the direction of $e(\theta)$ by replacing one of the roots of unity in the sum with another root of unity that is closer to $e(\theta)$. The claim follows easily from this observation.

There are $M d!$ choices for the maximal configurations of the $t_j$'s described in the previous paragraph. The values of $r$ and $B$ are the same for all of these configurations. Let $r_0$ and $B_0$ denote these values. Computing these explicitly, we find that
\[r_0 =\left|\sum_{l=1}^d e(\tfrac{l}{M})\right| = \frac{\sin(\pi d/M)}{\sin(\pi /M)},\]
and
\begin{align*}
B_0 =\frac{1}{M^d} \frac{1}{d!} \prod_{1 \leq j < k \leq d} |e(\tfrac{k}{M}) - e(\tfrac{j}{M})|^2 &= \frac{1}{M^d} \frac{1}{d!} \prod_{1 \leq j < k \leq d} 4\sin^2(\pi(k-j)/M) \\
&= \frac{1}{M^d} \frac{1}{d!}\prod_{m=1}^{d-1} \left(4\sin^2(\pi m/M)\right)^{d-m}.
\end{align*}
Putting everything together, we conclude that $S_n^{(d,L)}\sim M d! B_0 r_0^{2n}$ as $n \to \infty$, which is the desired result.
\end{proof}

A consequence of this theorem is that the Stanley--Wilf limit of the pattern set $\{d\cdots 1 (d+1), 1\cdots (L+1)\}$ is $\left(\frac{\sin(\pi d/M)}{\sin(\pi/M)}\right)^2$. Note that this quantity is unchanged when $d$ and $L$ are interchanged, which is consistent with the Wilf-equivalence established in Corollary \ref{cor:wilf}. One can also check that $C_{d,L}=C_{L,d}$, although this is less obvious from the stated formula.

\section{RS, $d$-RS, and skew $d$-RS growth diagrams} \label{sec:connections}
In this section we discuss an important special case of the growth diagrams considered in this paper. Namely, we consider the case where the sizes of adjacent partitions in these growth diagram can differ by at most 1. Any RSK, $d$-RSK, or skew $d$-RSK growth diagram that satisfies this restriction will be referred to as an \emph{RS}, \emph{$d$-RS}, or \emph{skew $d$-RS growth diagram}, respectively.

It is possible to give a direct characterization of these growth diagrams without reference to their Knuth-type generalizations. That is, they can be defined in terms of local rules that are more restrictive than the original local rules. We will call these the \emph{RS}, \emph{$d$-RS}, and \emph{skew $d$-RS local rules}, respectively.

To state these rules, let us first introduce some notation. Given a partition $\kappa$ and a positive integer $i$, we write $\kappa{+}(i)$ to denote the partition that is obtained from $\kappa$ by adding one to the $i$th part. We use the same notation for $d$-staircases. This carries the implicit assumption that performing this operation yields a valid partition or $d$-staircase.

The utility of this notation comes from the following observation: adjacent partitions (or $d$-staircases) that appear in RS, $d$-RS, or skew $d$-RS growth diagrams must either be equal or they are of the form $\kappa$ and $\kappa{+}(i)$. This follows from the fact that adjacent partitions (or $d$-staircases) in these growth diagrams are interlacing and their sizes differ by at most 1. Using this observation, it is straightforward to list all the valid cell configurations that can appear in any RS, $d$-RS, or skew $d$-RS growth diagram. These lists are what we will take as the definitions of the RS, $d$-RS, and skew $d$-RS local rules. These are shown in Figure \ref{fig:localrules}.

\begin{figure}[p]
\centering
\captionsetup[subfigure]{font=small,skip=6pt}
\setlength{\fboxsep}{4pt}
\setlength{\fboxrule}{0.4pt}
\begin{subfigure}{\textwidth}
\centering
\fbox{\begin{minipage}{\dimexpr\textwidth-2\fboxsep-2\fboxrule\relax}\centering
\begin{tikzpicture}[scale=1.15]
\begin{scope}[shift={(0,0)}]
    \draw (0,0) rectangle (1,1);
    \node at (.5,.5) {$0$};
    \node at (-.2,-.2) {$\kappa$};
    \node at (-.2,1.2) {$\kappa$};
    \node at (1.2,-.2) {$\kappa$};
    \node at (1.2,1.2) {$\kappa$};
\end{scope}
\begin{scope}[shift={(3,0)}]
    \draw (0,0) rectangle (1,1);
    \node at (.5,.5) {$0$};
    \node at (-.2,-.2) {$\kappa$};
    \node at (-.2,1.2) {$\kappa{+}(i)$};
    \node at (1.2,-.2) {$\kappa$};
    \node at (1.2,1.2) {$\kappa{+}(i)$};
\end{scope}
\begin{scope}[shift={(6,0)}]
    \draw (0,0) rectangle (1,1);
    \node at (.5,.5) {$0$};
    \node at (-.2,-.2) {$\kappa$};
    \node at (-.2,1.2) {$\kappa$};
    \node at (1.2,-.2) {$\kappa{+}(i)$};
    \node at (1.2,1.2) {$\kappa{+}(i)$};
\end{scope}
\begin{scope}[shift={(9,0)}]
    \draw (0,0) rectangle (1,1);
    \node at (.5,.5) {$0$};
    \node at (-.2,-.2) {$\kappa$};
    \node at (-.2,1.2) {$\kappa{+}(i)$};
    \node at (1.2,-.2) {$\kappa{+}(j)$};
    \node at (1.2,1.2) {$\kappa{+}(i){+}(j)$};
    \node at (.5,-.6) {Assuming $i \neq j$};
\end{scope}
\begin{scope}[shift={(1,-2.5)}]
    \draw (0,0) rectangle (1,1);
    \node at (.5,.5) {$0$};
    \node at (-.2,-.2) {$\kappa$};
    \node at (-.2,1.2) {$\kappa{+}(i)$};
    \node at (1.2,-.2) {$\kappa{+}(i)$};
    \node at (1.2,1.2) {$\kappa{+}(i + 1)$};
\end{scope}
\begin{scope}[shift={(8,-2.5)}]
    \draw (0,0) rectangle (1,1);
    \node at (.5,.5) {$1$};
    \node at (-.2,-.2) {$\kappa$};
    \node at (-.2,1.2) {$\kappa$};
    \node at (1.2,-.2) {$\kappa$};
    \node at (1.2,1.2) {$\kappa{+}(1)$};
\end{scope}
\end{tikzpicture}
\end{minipage}}
\caption{The RS local rule. In these configurations, $\kappa$ denotes a partition and $i,j$ denote natural numbers.}
\label{fig:rs_local_rule}
\end{subfigure}

\par\vspace*{1.5em}

\begin{subfigure}{\textwidth}
\centering
\fbox{\begin{minipage}{\dimexpr\textwidth-2\fboxsep-2\fboxrule\relax}\centering
\begin{tikzpicture}[scale=1.15]
\begin{scope}[shift={(0,0)}]
    \draw (0,0) rectangle (1,1);
    \node at (.5,.5) {$0$};
    \node at (-.2,-.2) {$\kappa$};
    \node at (-.2,1.2) {$\kappa$};
    \node at (1.2,-.2) {$\kappa$};
    \node at (1.2,1.2) {$\kappa$};
\end{scope}
\begin{scope}[shift={(3,0)}]
    \draw (0,0) rectangle (1,1);
    \node at (.5,.5) {$0$};
    \node at (-.2,-.2) {$\kappa$};
    \node at (-.2,1.2) {$\kappa{+}(i)$};
    \node at (1.2,-.2) {$\kappa$};
    \node at (1.2,1.2) {$\kappa{+}(i)$};
\end{scope}
\begin{scope}[shift={(6,0)}]
    \draw (0,0) rectangle (1,1);
    \node at (.5,.5) {$0$};
    \node at (-.2,-.2) {$\kappa$};
    \node at (-.2,1.2) {$\kappa$};
    \node at (1.2,-.2) {$\kappa{+}(i)$};
    \node at (1.2,1.2) {$\kappa{+}(i)$};
\end{scope}
\begin{scope}[shift={(9,0)}]
    \draw (0,0) rectangle (1,1);
    \node at (.5,.5) {$0$};
    \node at (-.2,-.2) {$\kappa$};
    \node at (-.2,1.2) {$\kappa{+}(i)$};
    \node at (1.2,-.2) {$\kappa{+}(j)$};
    \node at (1.2,1.2) {$\kappa{+}(i){+}(j)$};
    \node at (.5,-.6) {Assuming $i \neq j$};
\end{scope}
\begin{scope}[shift={(1,-2.75)}]
    \draw (0,0) rectangle (1,1);
    \node at (.5,.5) {$0$};
    \node at (-.2,-.2) {$\kappa$};
    \node at (-.2,1.2) {$\kappa{+}(i)$};
    \node at (1.2,-.2) {$\kappa{+}(i)$};
    \node at (1.35,1.2) {$\kappa{+}(i){+}(i{+}1)$};
    \node at (.5,-.6) {Assuming $i \leq d-1$};
\end{scope}
\begin{scope}[shift={(4.5,-2.75)}]
    \draw (0,0) rectangle (1,1);
    \node at (.5,.5) {$0$};
    \node at (-.2,-.2) {$\kappa$};
    \node at (-.2,1.2) {$\kappa{+}(d)$};
    \node at (1.2,-.2) {$\kappa{+}(d)$};
    \node at (1.2,1.2) {$\kappa{+}(d){+}(1)$};
\end{scope}
\begin{scope}[shift={(8,-2.75)}]
    \draw (0,0) rectangle (1,1);
    \node at (.5,.5) {$1$};
    \node at (-.2,-.2) {$\kappa$};
    \node at (-.2,1.2) {$\kappa$};
    \node at (1.2,-.2) {$\kappa$};
    \node at (1.2,1.2) {$\kappa{+}(1)$};
    \node at (.5,-.6) {Assuming $\kappa_d=0$};
\end{scope}
\end{tikzpicture}
\end{minipage}}
\caption{The $d$-RS local rule. Here $\kappa$ is required to be a $d$-partition and $i,j \leq d$. The Bloom--Saracino local rule is the same except that ``Assuming $\kappa_d=0$'' is not required in the last configuration.}
\label{fig:drs_local_rule}
\end{subfigure}

\par\vspace*{1.5em}

\begin{subfigure}{\textwidth}
\centering
\fbox{\begin{minipage}{\dimexpr\textwidth-2\fboxsep-2\fboxrule\relax}\centering
\begin{tikzpicture}[scale=1.15]
\begin{scope}[shift={(0,0)}]
    \draw (0,0) rectangle (1,1);
    \node at (.5,.5) {$0$};
    \node at (-.2,-.2) {$\kappa$};
    \node at (-.2,1.2) {$\kappa$};
    \node at (1.2,-.2) {$\kappa$};
    \node at (1.2,1.2) {$\kappa$};
\end{scope}
\begin{scope}[shift={(3,0)}]
    \draw (0,0) rectangle (1,1);
    \node at (.5,.5) {$0$};
    \node at (-.2,-.2) {$\kappa$};
    \node at (-.2,1.2) {$\kappa{+}(i)$};
    \node at (1.2,-.2) {$\kappa$};
    \node at (1.2,1.2) {$\kappa{+}(i)$};
\end{scope}
\begin{scope}[shift={(6,0)}]
    \draw (0,0) rectangle (1,1);
    \node at (.5,.5) {$0$};
    \node at (-.2,-.2) {$\kappa$};
    \node at (-.2,1.2) {$\kappa$};
    \node at (1.2,-.2) {$\kappa{+}(i)$};
    \node at (1.2,1.2) {$\kappa{+}(i)$};
\end{scope}
\begin{scope}[shift={(9,0)}]
    \draw (0,0) rectangle (1,1);
    \node at (.5,.5) {$0$};
    \node at (-.2,-.2) {$\kappa$};
    \node at (-.2,1.2) {$\kappa{+}(i)$};
    \node at (1.2,-.2) {$\kappa{+}(j)$};
    \node at (1.2,1.2) {$\kappa{+}(i){+}(j)$};
    \node at (.5,-.6) {Assuming $i \neq j$};
\end{scope}
\begin{scope}[shift={(1,-2.75)}]
    \draw (0,0) rectangle (1,1);
    \node at (.5,.5) {$0$};
    \node at (-.2,-.2) {$\kappa$};
    \node at (-.2,1.2) {$\kappa{+}(i)$};
    \node at (1.2,-.2) {$\kappa{+}(i)$};
    \node at (1.35,1.2) {$\kappa{+}(i){+}(i{+}1)$};
    \node at (.5,-.6) {Assuming $i \leq d-1$};
\end{scope}
\begin{scope}[shift={(4.5,-2.75)}]
    \draw (0,0) rectangle (1,1);
    \node at (.5,.5) {$0$};
    \node at (-.2,-.2) {$\kappa$};
    \node at (-.2,1.2) {$\kappa{+}(d)$};
    \node at (1.2,-.2) {$\kappa{+}(d)$};
    \node at (1.2,1.2) {$\kappa{+}(d){+}(1)$};
\end{scope}
\end{tikzpicture}
\end{minipage}}
\caption{The skew $d$-RS local rule. In these configurations, $\kappa$ is a $d$-staircase and $i,j \leq d$. The Elizalde local rule is the same except that $\kappa$ is a $(d,L)$-staircase, and the first three configurations are removed.}
\label{fig:elizalde_local_rule}
\end{subfigure}

\caption{Local rules appearing in Section~\ref{sec:connections}.}
\label{fig:localrules}
\end{figure}
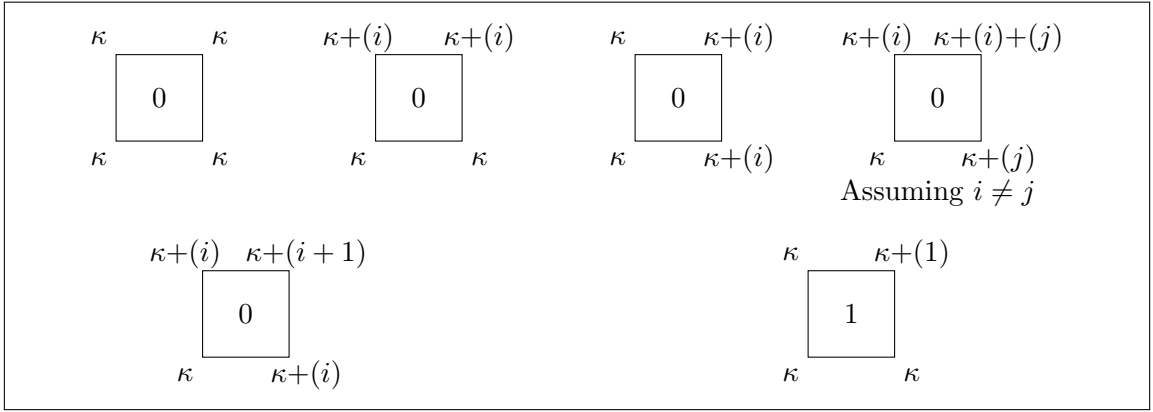
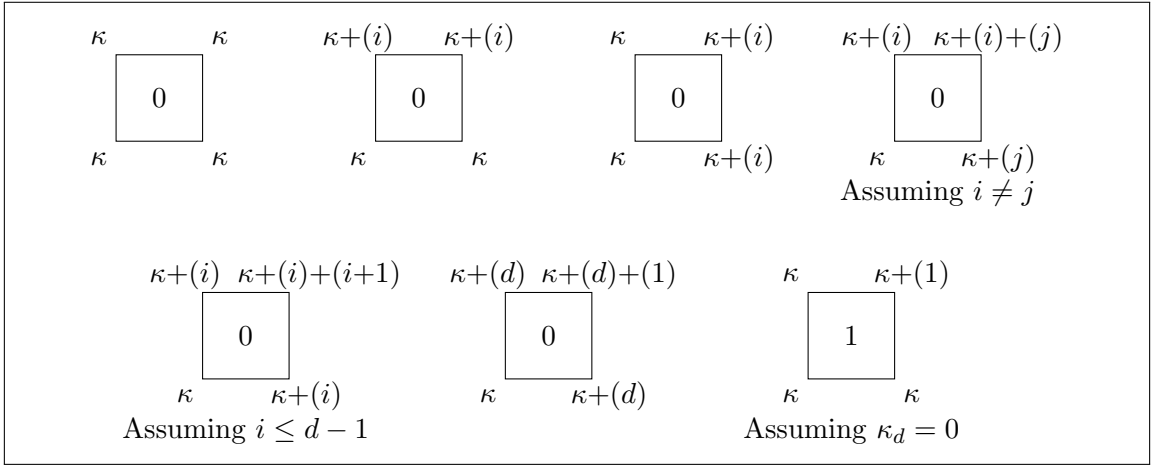
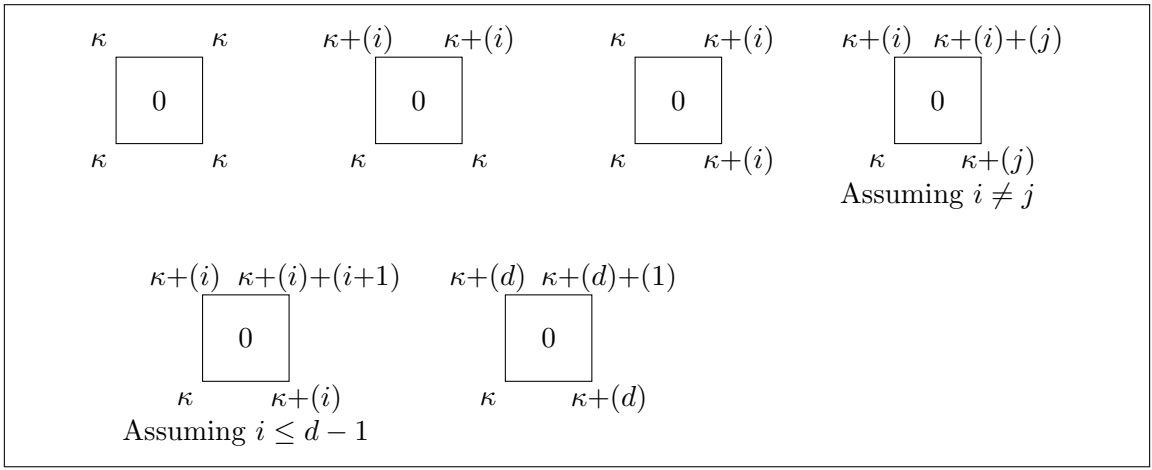

RS, $d$-RS, and skew $d$-RS growth diagrams lead to special cases of the correspondences that we have described in this paper. We call these the \emph{oscillating RS, $d$-RS, and skew $d$-RS correspondences}, respectively. In comparison with their Knuth-type generalizations, the tableaux are now restricted to those with weight vectors consisting of 0's and 1's. The fillings are restricted to those whose row and column sums are all 0 or 1.

The oscillating RS correspondence is well known (see \cite[Thm. 1]{krattenthaler2006}, \cite[Ch. 4]{roby1991}). The oscillating skew $d$-RS correspondence was essentially constructed by  Elizalde \cite{elizalde2025} as we discussed in Section \ref{subsec:skewdRSKconnections}. Elizalde was interested in the case of tableaux of that are $(d,L)$-cylindric and standard, but aside from this minor difference the configurations shown in Figure \ref{fig:elizalde_local_rule} are equivalent to those given in \cite[Sec. 5]{elizalde2025}. 

The oscillating $d$-RS correspondence appears in work of Bloom and Saracino \cite{bloom2012s}. However, they did not describe it exactly as we have. Bloom and Saracino considered a local rule that is more general than the $d$-RS local rule (see the caption of Figure \ref{fig:drs_local_rule}). They used this to construct a map from fillings of Young diagrams to oscillating tableaux. Let us call this the \emph{Bloom--Saracino map} (it is denoted by $\mathrm{seq}_k$ in \cite{bloom2012s} where $k=d+1$ in our notation). 

The domain of the Bloom--Saracino map is the set of all fillings of Young diagrams whose row and column sums are at most 1. It is not a bijective map, but it becomes a bijection (the oscillating $d$-RS correspondence) upon restricting to fillings that avoid $d\cdots 1 (d+1)$. This is not stated explicitly in \cite{bloom2012s}, but it is an immediate consequence of the main result given there. In fact, Bloom and Saracino's work goes further by giving a characterization of the fibers of the Bloom--Saracino map. In particular they show that the map is invariant under a certain transformation on fillings.

Bloom and Saracino's motivation for constructing their map was to give a growth diagram formulation of the Backelin--West--Xin bijection \cite{backelin2007} that we mentioned in Section \ref{sec:fillings_cors}. Their main result implies that the Backelin--West--Xin bijection is obtained by composing the oscillating RS and $d$-RS correspondences. Hence, the map that we constructed in Section~\ref{sec:fillings_cors} by composing the oscillating RSK and $d$-RSK correspondences is a natural extension of this.

In light of our work in this paper, it would be interesting to know whether Bloom and Saracino's results have some further significance in relation to cylindric tableaux. For example, it seems possible that a further study of the Bloom--Saracino map could be fruitful for developing a notion of Knuth equivalence for cylindric tableaux.

\section{Continuous piecewise-linear analogues} \label{sec:continuous}
The RSK-type correspondences that we have described in this paper also have continuous piecewise-linear analogues. For the classical RSK correspondence, such an analogue is well known (see \cite[Sec. 1]{bisi2021oz} for a discussion of this). We shall briefly describe how these analogues work and how they carry over to $d$-RSK (or skew $d$-RSK).

In the continuous piecewise-linear setting, we need to modify our earlier definitions to allow for real-valued quantities. For example, the fillings of Young diagrams are now allowed to have nonnegative \emph{real} entries. The notion of NE-chains and se-chains carry over to this setting in a natural way. In particular, the length of a NE-chain is now a continuous quantity (the sum of the entries in the chain), while the length of a se-chain remains discrete (the number of entries in the chain). The definition of ``avoids the pattern $d\cdots 1 (d+1)$'' remains the same. In particular $d$ is still a positive integer.

On the tableau side of the correspondences, we must replace integer partitions with \emph{real} partitions. We define these as weakly decreasing finite sequences of positive real numbers. The size of a real partition is now a continuous quantity (the sum of the parts), whereas the length remains discrete (the number of parts). The definition of interlacing carries over to real partitions verbatim. Hence, semistandard Young tableaux have a natural analogue in the continuous setting. The definition of $(d,L)$-interlacing also carries over verbatim, so $(d,L)$-cylindric semistandard Young tableaux also have a natural analogue in the continuous setting. The variable $L$ is now a continuous quantity.

The definitions of the RSK, $d$-RSK, and skew $d$-RSK local rules were written in terms of continuous piecewise-linear functions. Consequently, the proofs of our main results generalize with hardly any modification. This leads to continuous piecewise-linear analogues of the cylindric correspondences as well.

The most significant difference in this generalized setting is that there is no longer a natural notion of conjugation for real partitions. This means that row-strict tableaux do not fit naturally into the picture that we have just described. Interestingly, for RSK there is still a sort of dual correspondence---sometimes called the Burge correspondence---which can be defined in a continuous piecewise-linear fashion (see \cite[App. A]{betea2019}). It is unclear whether a cylindric analogue of the Burge correspondence exists. We leave this as an open question for the reader.

\bibliographystyle{amsplain} 
\bibliography{bibliography}

\end{document}